\documentclass[11pt, twoside]{article}
\usepackage{textcomp}
\usepackage[utf8]{inputenc}
\usepackage[english]{babel} 
\usepackage[T1]{fontenc} 
\usepackage[utf8]{inputenc} 
\usepackage{soul}
\usepackage[normalem]{ulem}
\usepackage{fancybox}
\usepackage{fancyvrb}
\usepackage{listings}
\usepackage{moreverb}
\usepackage{url}
\usepackage{graphicx}
\graphicspath{ {./images/} }
\usepackage{textcomp}
\usepackage{lmodern}
\usepackage{eurosym}
\usepackage[cyr]{aeguill}
\usepackage{amsmath}
\usepackage{wrapfig} 
\usepackage{makeidx} 
\usepackage{picinpar} 
\usepackage[final]{pdfpages} 
\usepackage{array,multirow,tabularx}
\usepackage{amssymb}
\usepackage[a4paper]{geometry}
\usepackage{amsthm}
\usepackage{tikz}
\usepackage{fancyhdr}
\usepackage{colortbl}
\usepackage{color}
\usepackage{setspace}
\usepackage{longtable}
\usepackage{hhline}
\usepackage{arydshln}
\usepackage{makecell}
\usepackage{mathtools}
\usepackage{tkz-tab}
\usepackage{tikz-cd}
\usepackage{multicol}
\usepackage{enumerate} 
\usepackage{hyperref}
\tikzset{node distance=2.0cm, auto}
\usepackage{quiver}

\usepackage{fancyhdr}
\renewcommand{\sectionmark}[1]{}
\renewcommand{\subsectionmark}[1]{}
\usepackage[section]{placeins}

\DeclareMathOperator{\ad}{ad}

\allowdisplaybreaks

\title{Structure and Representations of Restricted Lie color triple systems}
\author{Jon Beristain\footnote{Université de Haute-Alsace, IRIMAS UR 7499, F-68100 Mulhouse, France. \\ E-mail: \texttt{jon.beristain@uha.fr}. }  \; and  Abdenacer Makhlouf\footnote{Université de Haute-Alsace, IRIMAS UR 7499, F-68100 Mulhouse, France.\\ E-mail:  \texttt{abdenacer.makhlouf@uha.fr}.} }
\date{\today}

\begin{document}
	
	\newtheorem{thm}{Theorem}[section]
	\newtheorem{prop}[thm]{Proposition}
	\newtheorem{lem}[thm]{Lemma}
	\newtheorem{cor}[thm]{Corollary}
	\theoremstyle{definition}
	\newtheorem{defi}[thm]{Definition}
	\newtheorem*{rmq}{Remark}
	\newtheorem{ex}[thm]{Example}
	
	\maketitle
	\begin{abstract}
    The main purpose of this paper is to introduce and study restricted Lie color triple systems. We consider first restricted Lie color algebras and then provide a definition and relevant properties of restricted Lie color triple systems.  Moreover, we introduce and discuss representations of restricted Lie color algebras and restricted Lie color triple systems.
		
	\end{abstract}
	
	\tableofcontents
     \markboth{J. Beristain and A. Makhlouf}{Restricted Lie color triple systems}
     \markright{Restricted Lie color triple systems}
	%%%%%%%%%%%%%%%%%%%%%%%%%%%%%%%%%%%%%%%%%%%%%%%%%%%%%%%%%%%%%%%%%%%%%%%%%%%%%%%%%%%%%%
	
\section{Introduction}

Lie triple systems were introduced by Jacobson in order to study representations of the Jordan algebras \cite{J51} and then considered by Yamaguti who provides the current definition \cite{Y58}. Lie triple systems appeared also in geometry, they are related to symmetric spaces in the same way that Lie algebras are related to Lie groups; they correspond to  tangent spaces of symmetric spaces. It turns out that they are also relevant in elementary particle theory and  quantum mechanics theory, in particular in meson theory.  One of the main properties, shown by Jacobson,  is that  a Lie triple system is exactly the subspace of a Lie algebra closed under the ternary operation $[[\cdot,\cdot],\cdot] $. Indeed, a subspace $S$ of a Lie algebra $L$ such that $[[x,y],z] $ lies in $S$ for all $x,y,z$ in $S$ is a Lie triple system for this operation. In \cite{J51}, it was shown  that a Lie triple system $T$ can always be embedded into a Lie algebra $\mathcal{G} $ in such a way that $T$ is a subspace of $\mathcal{G} $ closed under the triple Lie product  $[[\cdot,\cdot],\cdot] $. The structure of Lie triple systems was  studied by Lister in \cite{L51} over a field of characteristic 0. In characteristic $p>0$, Lie algebras may have an additional structure called a restricted structure, a map $(\cdot)^{[p]}: L\to L $ that satisfies some conditions. Based on the observation that the algebra of derivations of a Lie algebra over a field a characteristic $p$ is closed under the Lie bracket $[D_1,D_2]=D_1D_2 - D_2D_1 $ and  thanks to the Leibniz rule we have that  $D^p$ is again a derivation for all derivations $D$,  Jacobson formally in \cite{J41} introduced a so-called restricted structure on a Lie algebra in positive characteristic. In \cite{H01}, Hodge extended  restricted structures to Lie triple systems with the aim of preserving the relationship between Lie algebras and Lie triple systems. Lie color algebras and Lie color triple systems are generalizations of Lie algebras and Lie triple systems, including Lie superalgebras and Lie super triple systems, see \cite{BMPZ92,S79}  for the general theory of color Lie algebra  and \cite{AAM13,CZC19,KZ23,LM22} for some relevant results on Lie color triple systems in characteristic 0. 

The aim of this paper is to introduce and study  restricted Lie color triple systems,  and generalize the relevant  results in \cite{H01} to the color setting over a field of characteristic $p$. Moreover, we tackle their representation theory. We define a representation of a Lie color triple system and a restricted representation of a restricted Lie color triple system, then show that the semi-direct product is endowed with a structure of a restricted Lie color triple system. Since  definitions of a Lie color algebra and a Lie color triple system  involve a bicharacter, the PBW-Theorem does not remain true if the characteristic of the base field is equal to $2$ or $3$. Throughout this paper, in order to keep some results true, specially the Jacobson's Theorem, the characteristic of the base field is assumed to be different from two and three (see \hyperref[rmq]{Remark~\ref*{rmq}}). 

%\subsection{Organization of this paper and main results} 

The paper is organized as follows. In Section \eqref{s2}, we  review some basic definitions and properties about Lie color algebras and Lie color triple systems. In Section \eqref{s3}, we recall the definition of restricted Lie color algebras and show  Jacobson's theorem for Lie color algebras (\hyperref[thm3.5]{Theorem~\ref*{thm3.5}}). In Section \eqref{s4}, we give a definition of a restricted Lie color triple system (\hyperref[d4.1]{Definition~\ref*{d4.1}}). We generalize some results in \cite{H01}, specially the fact that when the standard embedding of a Lie color triple system is restricted, then the Lie color triple system is restricted with the same $p$-map (\hyperref[c4.11]{Corollary~\ref*{c4.11}}). Conversely, we show that we can, under some assumptions, extend the $p$-map of a restricted Lie color triple system $T$ into an embedding of $T$: $L_D(T)$ (\hyperref[thm4.15]{Theorem~\ref*{thm4.15}}). Finally, we show Jacobson's Theorem for Lie color triple systems (\hyperref[thm4.16]{Theorem~\ref*{thm4.16}}). In Section \eqref{s5}, we define the concept pf restricted representation of a restricted Lie color triple system (\hyperref[d5.11]{Definition~\ref*{d5.11}}) and show that we can endow the semi-direct product $T\oplus V $ of a restricted Lie color triple system $T$ and its representation $(V,\theta) $ with a structure of restricted Lie color triple system (\hyperref[thm5.18]{Theorem~\ref*{thm5.18}}).

\section{Lie color algebras and Lie color triple systems}\label{s2}

In this section, we review some basic  definitions about Lie color algebras and Lie color triple systems over a field $\mathbf{K} $. Moreover, we provide some examples and relevant results and constructions.

\subsection{Lie color algebras}

\begin{defi}
    Let $(G,+)$ be an abelian group. A symmetric bicharacter $\varepsilon$ of $G$ is a map $\varepsilon:G\times G\to \mathbf{K}^* $ satisfying \begin{enumerate}
        \item $\varepsilon(\alpha+\beta,\gamma)= \varepsilon(\alpha,\gamma)\varepsilon(\beta,\gamma) $,
        \item $\varepsilon(\alpha,\beta+\gamma)= \varepsilon(\alpha,\beta)\varepsilon(\alpha,\gamma) $,
        \item $ \varepsilon(\alpha,\beta)\varepsilon(\beta,\alpha) =1$.
    \end{enumerate} We denote by $\mathcal{B}(G)$ the set of all bicharacters of $G$.
\end{defi}
It follows from the definition that $\varepsilon(0,\alpha)=\varepsilon(\alpha,0)=1 $ and $\varepsilon(\alpha,\alpha)= \pm 1 $.

\begin{ex}
    
    \begin{enumerate}
        \item The set $\mathcal{B}(\mathbb{Z}_2) $ has two elements $\varepsilon^+$ and $\varepsilon^-$ given by $\varepsilon^+(\bar{1},\bar{1})=1$, $\varepsilon^-(\bar{1},\bar{1})= -1 $.
        \item Let $\phi:G\to G'$ be a homomorphism of abelian groups. Let $\varepsilon \in \mathcal{B}(G')$. Define $\varepsilon^{\phi}: G\times G\to \mathbf{K}^* $ by $\varepsilon^{\phi}(\alpha,\beta)= \varepsilon(\phi(\alpha),\phi(\beta)) $. Then $\varepsilon^{\phi}\in \mathcal{B}(G) $.
    \end{enumerate}
\end{ex}
For further use, we denote by $G=G_+\sqcup G_- $, where $G_+= \left\{\gamma\in G, \varepsilon(\gamma,\gamma)=1 \right\}  $ and $G_-= \left\{\gamma\in G, \varepsilon(\gamma,\gamma)=-1 \right\} $. We note that $G_+ $ is a subgroup of $G$.

\begin{defi}
    Let $G$ be an abelian group. A   vector space $V$ is said to be a $G$-graded vector space if $V$ is  a direct sum of vector subspaces indexed by $G$, $V=\bigoplus\limits_{i\in G}V_i $. An element $x$ of $V$ is said homogeneous if there exists $i$ in $G$ such that $x$ lies in $V_i$. We denote by $\mathcal{H}(V)$ the set of homogeneous elements of $V$. For a homogeneous element $x\in V$, we denote by $\bar{x}\in G $ the degree of $x$.
\end{defi}

\begin{defi}[\cite{BMPZ92}]
    Let $G$ be an abelian group. A Lie color algebra is a triple $(L,[\cdot,\cdot],\varepsilon) $ where $L= \bigoplus\limits_{i\in G} L_i $ is a $G$-graded vector space, $[\cdot,\cdot]:L\times L\to L $ is a bilinear map and $\varepsilon $ is a bicharacter on $G$ such that for all homogeneous elements $u,v,x,y$ and $z$ in $\mathcal{H}(L)$, the following conditions are satisfied: \begin{enumerate}
        \item $[L_i,L_j]\subset L_{i+j} $, for all $i,j \in G$,
        \item  $[x,y]= -\varepsilon(\bar{x},\bar{y})[y,x] $,
        \item  $\varepsilon(\bar{z},\bar{x})[x,[y,z]]+ \varepsilon(\bar{x},\bar{y})[y,[z,x]]+ \varepsilon(\bar{y},\bar{z})[z,[x,y]]=0 $,
        
    \end{enumerate} where $\bar{x}\in G $ is the degree of $x$. 
    \end{defi}
    In the sequel of this section, $G$ is an abelian group, $\varepsilon$ is a fixed bicharacter and we will simply denote by $L$ the triple $(L,[\cdot,\cdot], \varepsilon) $.
For a Lie color algebra $L$, we denote by $L_+=\bigoplus\limits_{\gamma\in G_+}L_{\gamma} $ and $L_-=\bigoplus\limits_{\gamma\in G_-}L_{\gamma} $.

\begin{ex}[\cite{AAM13}] \label{Example1}
We consider 
%\begin{center}
    $\Gamma= \left\{(0,0,0),(1,1,0), (1,0,1), (0,1,1) \right\} \subset\mathbb{Z}_2^3 $ 
    %\end{center}
   and a  bicharacter $\varepsilon:\Gamma\times \Gamma\to \mathbf{K}^* $ defined by the matrix \begin{center}
       $(\varepsilon(i,j))_{i,j} =$ $ \begin{pmatrix}
1 & 1 & 1 & 1 \\       
1 & 1 & -1 &  -1 \\
1 & -1 & 1 & -1 \\
1 & -1 &  -1& 1  \\
\end{pmatrix}$,
    \end{center} where the elements $(0,0,0), (1,1,0), (1,0,1), (0,1,1) $ are ordered and numbered by the numbers $ 0,1,2,3$ respectively.

    We define a complex algebra $L$ with three generators $e_1,e_2$ and $e_3$ that satisfy the commutation relations $e_1e_2+e_2e_1= e_3 $, $e_1e_3+e_3e_1= e_2 $, $e_2e_3+e_3e_2=e_1 $. Let $L$ be a $\Gamma $-graded linear space $L= L_{(1,1,0)}\oplus L_{(1,0,1)} \oplus L_{(0,1,1)} $ with  $ L_{(1,1,0)}=\mathrm{span}\{e_1\} $, $L_{(1,0,1)}=\mathrm{span}\{e_2\} $, $L_{(0,1,1)}=\mathrm{span}\{e_3\} $. The homogeneous subspace of $L$ graded by $(0,0,0) $ is zero. We define on $L$ the bracket given, with respect to the basis $(e_1,e_2,e_3)$, by   \begin{center}
        $[e_1,e_1]=0 $, $[e_1,e_2]=e_3 $,

        $[e_2,e_2]=0 $, $[e_1,e_3]=e_2 $,

        $[e_3,e_3]=0 $, $[e_2,e_3]= e_1 $,
    \end{center} which makes $L$ into a three-dimensional Lie color algebra.
\end{ex}

\begin{defi}
    Let $L= \bigoplus\limits_{i\in G}L_i $ be a Lie color algebra. A linear map $D:L\to L$ is called a color derivation of degree $\alpha$ if for $x \in \mathcal{H}(L)$, $y\in L $, we have \begin{center}
        $D([x,y])= [D(x),y]+ \varepsilon(\alpha,\bar{x})[x,D(y)] $.
    \end{center} We denote by $\mathrm{Der}(L)_{\alpha} $ the vector space spanned by all color derivations of degree $\alpha $ and $\mathrm{Der}(L)= \bigoplus\limits_{\alpha\in G}\mathrm{Der}(L)_\alpha $.
\end{defi}
The Leibniz rule in the color case is  given by the following proposition:
\begin{prop}
    Let $L$ be a Lie color algebra and $D$ be a color derivation of degree $\alpha\in G_+$. Then, for all integers $n$ and  for all $x\in \mathcal{H}(L), y\in L ,$ we have \begin{center}
        $D^n([x,y])=\sum\limits_{k=0}^n \binom{n}{k}\varepsilon(\alpha,\bar{x})^k [D^{n-k}(x),D^k(y)] $.
    \end{center}
\end{prop}

\begin{proof}
    The case $n=0 $ is immediate. 
    Assume that the formula is true for an integer $n\geq 0$. Let $x\in \mathcal{H}(L), y\in L$. We have \begin{align*}
        D^{n+1}([x,y]) = &D( \sum\limits_{k=0}^n \tbinom{n}{k}\varepsilon(\alpha,\bar{x})^k [D^{n-k}(x),D^k(y)] ) \\  =& \sum\limits_{k=0}^n \tbinom{n}{k}\varepsilon(\alpha,\bar{x})^k ([D^{n-k+1}(x),D^k(y)] + \varepsilon(\alpha,\bar{x})[D^{n-k}(x),D^{k+1}(y)]) \\  = &[D^{n+1}(x),y]+ \sum\limits_{k=1}^n \tbinom{n}{k}\varepsilon(\alpha,\bar{x})^k [D^{n-k+1}(x),D^k(y)] \\ & + \sum\limits_{k=0}^{n-1} \tbinom{n}{k}\varepsilon(\alpha,\bar{x})^{k+1} [D^{n-k}(x),D^{k+1}(y)] + \varepsilon(\alpha,\bar{x})^{n+1} [x,D^{n+1}(y)] \\  =&  \sum\limits_{k=1}^n (\tbinom{n}{k}+\tbinom{n}{k-1}) \varepsilon(\alpha,\bar{x})^k [D^{n+1-k}(x),D^k(y)] +[D^{n+1}(x),y]\\ &+ \varepsilon(\alpha,\bar{x})^{n+1} [x,D^{n+1}(y)] \\ = & \sum\limits_{k=0}^{n+1} \tbinom{n+1}{k}\varepsilon(\alpha,\bar{x})^k [D^{n+1-k}(x),D^k(y)] ).
    \end{align*} 
    
\end{proof}

The universal enveloping algebra of a Lie color algebra $L= \bigoplus\limits_{\gamma\in G}L_\gamma $ is defined as in \cite{BMPZ92}. Let $T=T(L)$ be the tensor algebra of $L$. Then $T$ inherits a $G $-gradation given by the isomorphism \begin{center}
    $T\cong \bigoplus\limits_{\gamma\in G} (\bigoplus\limits_{\left\{\gamma_1,\cdots,\gamma_n, n\in \mathbf{N}\mid \gamma_1+\cdots+\gamma_n= \gamma   \right\} }L_{\gamma_1}\otimes\cdots\otimes L_{\gamma_n}).  $ \end{center} 
    We consider the ideal $I$ (of the associative algebra) of $T$ generated by all tensors of the form $u\otimes v - \varepsilon(\bar{u},\bar{v})v\otimes u -[u,v]$ for homogeneous elements $u,v$ and set $U(L)= T(L)/I $. We have a homomorphism of Lie color algebra $\iota : L\to U(L) $. Thanks to  the universal property of the tensor algebra, we have that for all color associative algebra $A$ with an embedding $\varphi:L\to A_{L} $ there exists a unique graded homomorphism of associative algebra $f:U(L)\to A $ such that $f\circ \iota= \varphi $. 
    Hence $U(L)$ is called the universal enveloping algebra of $L$. We have a colored version of the PBW-Theorem in characteristic different from $2$ and $3$, and so in this case we have  $\iota:L\hookrightarrow U(L) $.

\begin{rmq}[\cite{F01}]\label{rmq}
In characteristic $2$ and $3$, $U(L)$ is still the universal enveloping algebra of a Lie color algebra $L$ but the embedding $\iota:L\to U(L) $ is no longer injective in general. Indeed, if the characteristic is $2$, given an associative color algebra $A$, we have for the associated Lie color algebra that \begin{center}
    $[a,a]=0 $ for all $a\in A$.
\end{center} However, we do not necessary have $[x,x]=0$ for all $x$ in a Lie color algebra if the characteristic is $2$. Similarly, if the characteristic is $3$, we have that \begin{center}
    $[[a,a],a]=0 $ for all $a\in A$.
\end{center} However, we do not necessary have $[[x,x],x] $ for all $x$ in a Lie color algebra if the characteristic is $3$. In these cases, there are no possible injective embeddings.
\end{rmq}

\begin{rmq}[\cite{BMPZ92}]
    In characteristic 3, if in the definition of Lie color algebra we assume in addition that $[y,[y,y]]=0 $ for all homogeneous elements $y\in L_-$, then the PBW-Theorem remains true.
\end{rmq}

\subsection{Lie color triple systems}

A definition of Lie color triple system can be found in \cite{CZC19}. According to the authors, it was originally given in  Zeng's thesis, which is  unfortunately not available.

\begin{defi}\label{d2.8}
    A Lie color triple system is a triple $(T,[\cdot,\cdot,\cdot],\varepsilon) $ consisting of a $G$-graded vector space $T=\bigoplus\limits_{i\in G}T_i$, a trilinear map $[\cdot,\cdot,\cdot]:T\times T\times T\to T $ and a bicharacter $\varepsilon $ on $G$ such that for all $u,v,x,y$ and $z$ in $\mathcal{H}(T)$, the following conditions are satisfied: \begin{align}
       & [T_i,T_j,T_k]\subset T_{i+j+k} , \text{ for all } i,j,k \in G,\label{lts1}\\
        & [x,y,z]= -\varepsilon(\bar{x},\bar{y})[y,x,z] ,\label{lts2}\\
       & \varepsilon(\bar{z},\bar{x})[x,y,z]+ \varepsilon(\bar{x},\bar{y})[y,z,x]+ \varepsilon(\bar{y},\bar{z})[z,x,y]=0 ,\label{lts3}\\
       & [u,v,[x,y,z]]=[[u,v,x],y,z]+ \varepsilon(\bar{u}+\bar{v},\bar{x})[x,[u,v,y],z] + \varepsilon(\bar{u}+\bar{v},\bar{x}+\bar{y})[x,y,[u,v,z]] \label{lts4}. 
    \end{align} 
    Morphisms of Lie color triple systems are defined in a very natural way.  In the sequel, we will denote by $T $  a Lie color triple system instead of the triple, when there is no ambiguity.
\end{defi}

\begin{ex}
\begin{enumerate}
    \item Let $L$ be a Lie color algebra. Then any subspace $T$ of $L$ closed under the operation $(x,y,z) \mapsto[[x,y],z] $ for all $x,y,z \in T $ is a Lie color triple system denoted by $L_{\mathrm{trip}} $.
    \item Consider Example \ref{Example1}, $L=L_{(1,1,0)}\oplus L_{(1,0,1)} \oplus L_{(0,1,1)} $, then $L $ with the bracket $[x,y,z]=[[x,y],z] $ is a Lie color triple system. It is defined by the following ternary brackets:\begin{center}
        $[e_i,e_i,e_j]=0 $ for all $i,j\in \left\{1,2,3 \right\}$, $[e_1,e_2,e_1]=e_2 $, $[e_1,e_2,e_2]=e_1 $,

        $[e_1,e_2,e_3]=0 $, $[e_1,e_3,e_1]=e_3 $, $[e_1,e_3,e_2]=0 $, $[e_1,e_3,e_3]= e_1 $,

        $[e_2,e_3,e_1]=0 $, $[e_2,e_3,e_2]=e_3 $, $[e_2,e_3,e_3]=e_2 $.
    \end{center}
    \end{enumerate}
\end{ex}

\begin{defi}[\cite{KZ23}]
    Let $T= \bigoplus\limits_{i\in G}T_i $ be a Lie color triple system. Then a linear map $D:T\to T$ is a color derivation of degree $\alpha$ if for all $x,y \in \mathcal{H}(T), z\in T$, we have \begin{center}
        $D([x,y,z])= [D(x),y,z]+ \varepsilon(\alpha,\bar{x})[x,D(y),z]+ \varepsilon(\alpha,\bar{x}+\bar{y})[x,y,D(z)] $,
    \end{center} and if $D(T_\gamma)\subset T_{\alpha+\gamma}$. We denote by $\mathrm{Der}(T)_{\alpha} $ the vector space spanned by all the color derivations of degree $\alpha $ and $\mathrm{Der}(T)= \bigoplus\limits_{\alpha\in G}D_\alpha $.
\end{defi}

\begin{thm}[\cite{KZ23}]
    Let $T$ be a Lie color triple system. Then $\mathrm{Der}(T)$ equipped with the bracket $[D_i,D_j]= D_iD_j - \varepsilon(i,j)D_jD_i $ for $D_i\in \mathrm{Der}(T)_i,D_j\in \mathrm{Der}(T)_j $, is a Lie color algebra.
\end{thm}

For $x,y$ in $\mathcal{H}(T) $, we set $D_{x,y} $ as the endomorphism of $T $ defined by %\begin{center}
    $D_{x,y}(z)=[x,y,z] $.
%\end{center} 
Then $D_{x,y} $ is a color derivation of degree $\bar{x}+\bar{y} $.  The subspace in $\mathrm{Der}(T)$ spanned by all the $D_{x,y} $ form the inner derivations and is denoted by $\mathrm{InnDer}(T)$.

\begin{defi}[\cite{KZ23}]
    Let $T$ be a Lie color triple system and $H $ be a $G $-graded subspace of $T$. Then $H$ is called a color ideal of $T$ if $[H,T,T]\subset H $.
\end{defi}

\begin{defi}
    Let $T$ be a Lie color triple system. The center of $T$ is defined by \begin{center}
        $Z(T)=\left\{x\in T \mid [x,y,z]=0, \forall y,z\in T \right\}  $.
    \end{center}
\end{defi}

\begin{prop}
    Let $T$ be a Lie color triple system, then the center of $T$ is a color ideal of $T$.
\end{prop}

\begin{thm}[\cite{KZ23}]
  The set   $\mathrm{InnDer}(T)$ is a color ideal of $\mathrm{Der}(T)$.
\end{thm}

\begin{defi}
    Let $T$ be a Lie color triple system. An embedding of $T$ is a Lie color algebra $L$  together with a linear map $\varphi: T\to L $ such that for all $x,y,z\in T $, $$\varphi([x,y,z])= [[\varphi(x),\varphi(y)],\varphi(z)]. $$
\end{defi}

Let $T= \bigoplus\limits_{i\in G}T_i $ be a Lie color triple system and consider the vector spaces \begin{eqnarray*}
  && L_D(T)= \bigoplus\limits_{i\in G}L_D(T)_i,  \text{ where } L_D(T)_i= T_i \oplus \mathrm{Der}(T)_i \\ && \text{and }L_S(T)=\bigoplus\limits_{i\in G} L_S(T)_i, \text{  where }  L_S(T)_i= T_i\oplus \mathrm{InnDer}(T)_i .  
\end{eqnarray*}
Define a bilinear product on $L_D(T)$ as follows, for $x,y$ in $\mathcal{H}(T)$, $D,D'\in \mathcal{H}(\mathrm{Der}(T))$ \begin{align*}
[x,y]= & D_{x,y} \\
[D,x]=& D(x) \\
[x,D]=&-\varepsilon(\bar{x},\bar{D})D(x) \\
[D,D']= &DD'- \varepsilon(\bar{D},\bar{D'})DD'.
\end{align*} With this bracket, $L_D(T) $ is a Lie color algebra.

\begin{prop}
    $L_S(T) $ is a color ideal of $L_D(T) $.
\end{prop}

\begin{thm}[\cite{KZ23}]
   The set  $L_S(T) $ carries a Lie color algebra structure and it is an injective embedding of $T$, called the standard embedding of $T$. It has the property $[x,y]=0 $ for $x,y$ in $T\subset L_S(T)$ if and only if, $[x,y,z]=0 $ for all $z\in T $. Moreover, we have  $L_S(T)= T \oplus [T,T] $.
\end{thm}

Let $T= \bigoplus\limits_{\gamma\in G}T_\gamma$ be a Lie color triple system. An embedding $(L,\varphi) $ of $T$ is called a universal embedding if for any embedding $\varphi':T\to L' $, there exists a unique Lie algebra homomorphism $\psi:L\to L' $ such that $\psi\circ\varphi= \varphi' $. We see from this definition that if there exists a universal embedding of $T$ then it is unique. Now, we will construct the universal enveloping color Lie algebra $L_U(T) $ of a Lie color triple system, for a base field of characteristic greater than three. For that, we use a Witt's Theorem  (\cite[Theorem $I.3.1 $]{S67},\cite[p.541]{H01}). We  construct the free color Lie algebra of $T$ using the tensor algebra of $T$ according to the PBW-Theorem which is valid if the characteristic is greater than $3$ (see \cite[p.85]{BMPZ92}).  Let $A_T = \bigoplus\limits_{n\in \mathbf{N}}T^{\otimes n} $ be the tensor algebra of $T$. Then $A_T$ inherits a $G$-graduation given by the isomorphism \begin{center}
    $A_T\cong \bigoplus\limits_{\gamma\in G} (\bigoplus\limits_{\left\{\gamma_1,\cdots,\gamma_n, n\in \mathbf{N}\mid \gamma_1+\cdots+\gamma_n= \gamma   \right\} }T_{\gamma_1}\otimes\cdots\otimes T_{\gamma_n}).  $  \end{center} Then $A_T$ is a Lie color algebra with the bracket $[x,y]= x y- \varepsilon(\bar{x},\bar{y})y x $ for all $x,y \in \mathcal{H}(A_T)$. Let $\mathcal{L} $ be the Lie subalgebra generated by the elements of $T$ inside $A_T$. Let $J$ be the ideal (of the Lie algebra) of $\mathcal{L}$ generated by elements of the form $[[x,y],z]-[x,y,z] $ for homogeneous elements in $T$ and let $L_U(T)=\mathcal{L}/J $. Then the canonical map $T\to L_U(T)$ is an embedding and $L_U(T) $ is the universal enveloping Lie algebra of $T$.

Since we have $T\hookrightarrow L_S(T) $, the universal property of $L_U(T)$ implies that $T\hookrightarrow L_U(T)$. 
\newline

From now on, we suppose that $T$ is already embedded in its standard embedding $L_S(T)$. Moreover, since the PBW-Theorem is not valid if the characteristic is $2$ or $3$, we will restrict ourselves to the characteristic greater than $3$, even if it is not necessary for all our results.

\begin{rmq}
    Assume that the characteristic of the base field $\mathbf{K} $ is $3$. The PBW-Theorem remains true if we assume in the definition of a Lie color algebra that we have in addition  $[y,[y,y]]=0 $ for all homogeneous elements $y\in L_-$. In order to keep the link between Lie color algebras and Lie triple systems, we assume for the definition of a Lie triple system that in addition $[y,y,y]=0 $ fo all homogeneous elements $y\in T_-$. As the universal property of our construction of the universal embedding is a consequence of the PBW-Theorem, it is still the universal embedding of a Lie triple system in characteristic $3$ with this assumption. Then, all the results of this paper remain true with this assumption in characteristic $3$.   
\end{rmq}
\section{Restricted Lie color algebras and Restricted Lie color triple systems}
In this section, we recall some basics about restricted Lie color algebras and introduce the concept of restricted Lie color triple systems. We discuss Jacobson's Theorem in color case and the embedding of restricted Lie color triple systems into restricted Lie color algebras.
\subsection{Restricted Lie color algebras}\label{s3}
In this section, we recall the definition of a restricted Lie color algebra. Then, we show that Jacobson's theorem remains true in this case.

\begin{defi}[\cite{F01}]
    We say that $L$ is a restricted Lie color algebra over a field $\mathbf{K}$ of characteristic $p>3$ if $L$ is a Lie color algebra equipped with a set of $p$-maps $(\cdot)^{[p]}: L_\gamma\to L_{p\gamma}$ satisfying, for all $\gamma \in G_+$, the following properties: \begin{enumerate}
        
        \item $(\alpha x)^{[p]}= \alpha^p x^{[p]} $, for all $\alpha \in \mathbf{K}$ and $x\in L_\gamma$,
        \item  $[x^{[p]},y]= (ad_x)^p(y) $, for all $x\in L_\gamma $ and $y\in L $,
        \item  $(x+y)^{[p]}= x^{[p]} + y^{[p]} + \sum\limits_{i=1}^{p-1} s_i(x,y) $, for all $ x,y\in L_\gamma$,
    \end{enumerate} where $\ad_x(y) = [x,y] $ and $is_i(x,y) $ is the coefficient of $\lambda^{i-1} $ in $(\ad_{\lambda x+y})^{p-1}(x) $. We denote by $(\cdot)^{[p]}:L_+\to L_+ $ the set of all the $p$-maps for $\gamma \in G_+ $.
\end{defi}

\begin{defi}
    Let $(L,(\cdot)^{[p]}), (L',(\cdot)^{[p]'})$ be two restricted Lie color algebras. A morphism of restricted Lie color algebras between $L$ and $L'$ is a morphism of Lie color algebras $f:L\to L' $ such that $f(x^{[p]})=f(x)^{[p]'} $ for all $x\in L$.
\end{defi}

\begin{defi}
    Let $V$ be a vector space. We say that a map $f:V\to V $ is $p$-semi-linear if $f(\alpha x+ y)= \alpha^p f(x)+ f(y) $ for all $x,y\in V$, $\alpha\in \mathbf{K}$.
\end{defi}

\begin{ex}
    \begin{enumerate}
        \item Let $A$ be a color associative algebra. Then, the color Lie algebra associated is restricted by defining for all $x\in A_\gamma $, $x^{[p]}=x^p $.
        \item Let $L$ be an abelian Lie color algebra with maps $f_\gamma:L_\gamma\to L_{p\gamma} $. Then, $L$ is restricted if and only if $f_\gamma $ is $p$-semi-linear for all  $\gamma\in G_+ $.
    \end{enumerate}
\end{ex}

For the next results, we follow the proofs presented in \cite{SF88} and \cite{E23}.
Let $L$ be a Lie color algebra and $S\subset L$ be a subset of $L$. We define the centralizer of $S$ in $L$ as \begin{center}
    $C_L(S)=\left\{ x\in L, [x,s]=0 \; \forall s\in S \right\} $.
\end{center}

\begin{prop}\label{p3.4}
    Let $L$ be a  Lie color subalgebra of a restricted Lie color algebra $H$ and let $(\cdot)^{[p]_1}:L_\gamma\to L_{p\gamma} $ be a map for all $\gamma \in G_+ $. Then, $(\cdot)^{[p]_1} $ is a $p$-map if and only if there exists a $p$-semi-linear map of degree $p$ $f_\gamma:L_\gamma\to C_H(L)$ such that $(\cdot)^{[p]_1}=(\cdot)^{[p]}+f $ for all $\gamma\in G_+$.
\end{prop}

\begin{proof}
     Assume that $(\cdot)^{[p]_1} $ is a $p$-map and  let $\gamma\in G_+$. Then for all $x\in L_\gamma$, $x^{[p]}-x^{[p]_1}\in C_H(L) $. It follows that $f_\gamma(x):= x^{[p]}-x^{[p]_1} $ is $p$-semi-linear. Conversely, assume that there exists a $p$-semi-linear map $f_\gamma: L_\gamma\to C_H(L)$ such that $(\cdot)^{[p]_1}= (\cdot)^{[p]}+ f_\gamma $. Then for $x,y\in L_\gamma $, we have \begin{align*}
        (x+y)^{[p]_1} &= (x+y)^{[p]}+ f_\gamma(x+y) \\ &= x^{[p]}+y^{[p]} + \sum\limits_{i=1}^p s_i(x,y) + f_\gamma(x) +f_\gamma(y) \\ & = x^{[p]_1}+y^{[p]_1} + \sum\limits_{i=1}^ps_i(x,y). 
    \end{align*} The other condition for being a $p$-map is straightforward. 
\end{proof}
In the following, we prove the color version of Jacobson's Theorem.
\begin{thm}\label{thm3.5}
    Let $L$ be a Lie color algebra over a field $\mathbf{K} $ of characteristic greater than $3$. Assume that for all $\gamma\in G_+ $, there exists $(e_{j,\gamma})_{j\in J_\gamma} $ a basis of $L_\gamma$ such that there are $y_{j,\gamma}\in L_{p\gamma} $ with $(\ad (e_{j,\gamma}))^p= \ad (y_{j,\gamma})$ for all $j\in J_\gamma $. Then there exists exactly one $p$-map $(\cdot)^{[p]}:L_+\to L_+ $ such that $e_{j,\gamma}^{[p]}=y_{j,\gamma} $ for all $j\in J$. 
\end{thm}

\begin{proof}
    Let $\gamma\in G_+ $. The Lie color algebra $L$ is embedded in its universal enveloping algebra $U(L)$, which is an associative algebra and so the associated Lie algebra  is restricted. Thus, in $U(L)$, $\ad_{e_{j,\gamma}}^p= \ad_{e_{j,\gamma}^p} $ and it follows that $e_{j,\gamma}^p-y_{j,\gamma}\in C_{U(L)}(L) $. We define a map $f_\gamma: L_\gamma\to C_{U(L)}(L) $, $f_\gamma(\sum\alpha_je_j):= \sum\alpha_j^p(y_{j,\gamma}-e_{j,\gamma}^p)$, where $\alpha_j\in \mathbf{K} $. Consider \begin{center}
        $ V = \left\{x\in L_\gamma , x^p+f_\gamma(x) \in L_{p\gamma}\right\} $.
    \end{center} Since $(\alpha x+y)^p+f_\gamma(\alpha x+y) = \alpha^px^p+y^p+\sum\limits_{i=1}^ps_i(\alpha x,y) + \alpha^pf_\gamma(x)+ f_\gamma(y) \in L\cap U(L)_{p\gamma}= L_{p\gamma} $, we have that $V$ is a vector space. Since $(e_{j,\gamma})_{j\in J}$ lies in $V$, $V=L_\gamma$. Then, according to Proposition~\ref{p3.4},
    %\hyperref[p3.4]{Proposition~\ref*{p3.4} },
    we see that $ (\cdot)^{[p]_1}:= x\mapsto x^p+f_\gamma(x) $ is a $p$-map of $L$ and hence $L$ is restricted.
\end{proof}

\begin{ex}
   Consider  $L= L_{(1,1,0)}\oplus L_{(1,0,1)} \oplus L_{(0,1,1)} $ defined in Example \ref{Example1}. We endow $L$ with a $p$-map using the Jacobson's Theorem. We can see that, since $p$ is odd, $\ad_{e_1}^p(e_1)=0 $, $\ad_{e_1}^p(e_2)=e_3 $ and $\ad_{e_1}^p(e_3)= e_2 $ so for all $y\in L$, $\ad_{e_1}^p(y)= [e_1,y] $. As $[e_2,e_1] = -\varepsilon(2,1)[e_1,e_2]=[e_1,e_2]=e_3 $ and $[e_2,e_3]=e_1 $, it follows that for all $y\in L $, $\ad_{e_2}^p(y)=[e_2,y] $. Similarly, for all $y\in L$, $\ad_{e_3}^p(y)=[e_3,y] $. Consequently, the identity on the basis is a $p$-map according to the Jacobson's Theorem. The center of $L$ is trivial so this $p$-map is unique. 
\end{ex}

\subsection{Restricted Lie color triple systems}\label{s4}
In this section, we define restricted Lie color triple systems and generalize some results of \cite{H01} to the color case. In particular, we see that if the standard embedding $L_S(T) $ of a Lie color triple system is restricted, so is $T$ and conversely, under some assumptions if $T$ is a restricted Lie color triple system, so is $L_D(T) $. \newline

Let $T$ be any Lie color triple system and $n\geq 3$ any positive odd integer. For elements $x_1,x_2,\dots , x_n $ in $T$, define \begin{center}
    $(x_1,x_2,\dots,x_n)= [[\cdots [[x_1,x_2,x_3],x_4,x_5],\cdots],x_{n-1},x_n] \in T$.
\end{center} For $\gamma\in G_+ $, $x,y\in T_\gamma $, the expression $(x,Zx+y,Zx+y,\cdots , Zx+y) $, with $Zx+y$ occurring $p-1$ times, is a polynomial in $Z$ with coefficients in $T$. For $i=1,\cdots,p $, we define $s_i(x,y) \in T$ by requiring $is_i(x,y)$ to be the coefficient of $Z^{i-1} $ in $(x,Zx+y,Zx+y,\cdots,Zx+y) $. For any embedding $L$ of $T$, we have \begin{center}
    $(x,Zx+y,Zx+y,\cdots,Zx+y)= [Zx+y,[Zx+y,[\cdots ,[Zx+y,x]\cdots]]] = (ad(Zx+y))^{p-1}(x) $.
\end{center} In this case, $is_i(x,y) $ is the coefficient of $Z^{i-1} $ in the expression $((\ad(Zx+y))^{p-1}(x) $, which agrees with the use of the notation $s_i(x,y) $ in the definition of a restricted Lie color algebra.

For a Lie color triple system $T$, we denote by $T_+=\bigoplus\limits_{\gamma\in G_+}T_\gamma $ and $T_-= \bigoplus\limits_{\gamma\in G_-}T_\gamma $.

\begin{defi}\label{d4.1}
    A Lie color triple system $T$ over a field of characteristic $p>3$ is restricted if it is given a set of $p$-maps $(\cdot)^{[p]}:T_\gamma\to T_{p\gamma} $ such that the following conditions are satisfied for all $\gamma\in G_+ $ : \begin{enumerate}
        
        \item $(\alpha x)^{[p]}= \alpha^p x^{[p]} $ for all $\alpha$ in $\mathbf{K}$ and all $x$ in $T_\gamma$,
        \item $(x+y)^{[p]}= x^{[p]}+y^{[p]}+ \sum\limits_{i=1}^{p-1}s_i(x,y) $ for all $x,y$ in $T_\gamma$,
        \item  $[x,y^{[p]},z]=(x,y,\cdots,y,z) $($p$ copies of y) for all $y$ in $T_\gamma$ and all $x,z$ in $T$.
    \end{enumerate}
\end{defi}

\begin{defi}
    Let $(T,(\cdot)^{[p]}), (T',(\cdot)^{[p]'})$ be two restricted Lie color triple systems. A morphism of restricted Lie color triple systems between $T$ and $T'$ is a morphism of Lie color triple systems $f:T\to T' $ such that $f(x^{[p]})=f(x)^{[p]'} $ for all $x\in T$.
\end{defi}

\begin{ex}
\begin{enumerate}
    \item   Let $L$ be a restricted Lie color algebra. Then $L$ equipped with the bracket $[x,y,z]=[[x,y],z] $ is a restricted Lie color triple system.
    \item We have seen that we can endow $L= L_{(1,1,0)}\oplus L_{(1,0,1)} \oplus L_{(0,1,1)}  $ with a restricted structure given by the identity on the basis and that this structure is unique. So the identity on the basis make $L $ into a restricted Lie color triple system. The center of $L$ as a Lie color algebra is trivial so is the center of $L$ as a Lie color triple system. Then the identity on the basis is the unique  $p$-map.
    \end{enumerate}
\end{ex}

\begin{prop}
    Let $T$ be a restricted Lie color triple system. Let $\gamma\in G_+ $. For each $x\in T_\gamma\subset L_S(T)_\gamma$, we have that $\ad_x\in Der(L_S(T)) $ satisfies $\ad_x^p= \ad_{x^{[p]}} $.
\end{prop}

\begin{proof}
    Since $L_S(T)= T\oplus [T,T] $, it is enough to show that the maps $(\ad_x)^p $ and $\ad_{x^{[p]}} $ agree with $T$ and so on $\mathcal{H}(T) $. Let $y\in \mathcal{H}(T)$. Since $p $ is odd, we have \begin{align*}
        (\ad_x)^p(y)= [x,[\cdots[x,y]]\cdots]&= -\varepsilon(p\bar{x},\bar{y})\varepsilon(\bar{x},\bar{x})^{\frac{p(p-1)}{2}}(y,x,\cdots,x,x) \\ &= -\varepsilon(p\bar{x},\bar{y})(y,x,\cdots,x,x) \in [T,T]
    \end{align*} and  $\ad_{x^{[p]}}(y)= [x^{[p]},y]= -\varepsilon(p\bar{x},\bar{y})[y,x^{[p]}] $. Then, for all $z\in T$, we have $[\ad_x^p(y),z]= -\varepsilon(p\bar{x},\bar{y})(y,x,\cdots,x,z) $ and $[\ad_{x^{[p]}}(y),z]= -\varepsilon(p\bar{x},\bar{y})[y,x^{[p]},z]= -\varepsilon(p\bar{x},\bar{y})(y,x,\cdots,x,z) $. So $(\ad_x)^p(y)$ and $\ad_{x^{[p]}}(y) $ agree with the  definition of $L_S(T)$.
\end{proof}

\begin{cor}
   We have $[x,y,z^{[p]}]=(x,y,z,\cdots,z) $ for all $x,y\in T, z\in T_\gamma$ for all $\gamma \in G_+$.
\end{cor}

\begin{proof}
Let $\gamma\in G_+, x,y\in \mathcal{H}(T), z\in T_\gamma $
    By applying the identity $(\ad_z)^p=\ad_{z^{[p]}} $ to $[x,y]\in [T,T] $, we have \begin{center}
        $[z^{[p]},[x,y]]=[z,[\cdots[z,[x,y]]]\cdots] $.
    \end{center} Since $[z^{[p]},[x,y]]= -\varepsilon(p\bar{z},\bar{x}+\bar{y})[x,y,z^{[p]}] $, we have \begin{center}
        $[z,[\cdots[z,[x,y]]]\cdots]= -\varepsilon(p\bar{z},\bar{x}+\bar{y})\varepsilon(\bar{z},\bar{z})^{\frac{p(p-1)}{2}}(x,y,z,\cdots,z)=  -\varepsilon(p\bar{z},\bar{x}+\bar{y})(x,y,z,\cdots,z)$.
    \end{center} Thus the result.
\end{proof}

\begin{prop}
    Let $L$ be a restricted Lie color algebra and $\theta$ be an involution of degree $0$ of $L$ as a restricted Lie color algebra. Then $T=\left\{x \in L, \theta(x)=-x \right\} $ is a restricted Lie color triple system.
\end{prop}

\begin{proof}
Let $x,y,z\in T$.
   Since $\theta([x,y,z])=\theta([[x,y],z])=[\theta(x),\theta(y),\theta(z)] = -[x,y,z] $, $T$ is closed under this bracket and $T$ is a Lie color triple system. In addition, $T$ is closed under the $p$-map. Indeed, let $\gamma\in G_+,  x \in L_\gamma$, we have  $\theta(x^{[p]})=\theta(x)^{[p]}=(-1)^px^{[p]}=-x^{[p]} $ because $p$ is odd. Let $y\in T_\gamma$, $x,z\in \mathcal{H}(T)$, in order to prove Condition $3. $ of  \hyperref[d4.1]{Definition ~\ref*{d4.1}}, just prove it for all $x,z\in \mathcal{H}(T)$. Let $x,z\in \mathcal{H}(T) $, we have \begin{center}
        $[x,y^{[p]},z]=[[x,y^{[p]}],z]=-\varepsilon(\bar{x},\bar{y^{[p]}})[[y^{[p]},x],z]= -\varepsilon(\bar{x},\bar{y})^p[[y^{[p]},x],z] $
    \end{center} and \begin{center}
        $[y^{[p]},x]=(\ad_y)^p(x)= [y,[y,[\cdots[y,x]]]\cdots]= -\varepsilon(\bar{y},\bar{x})^p(x,y,\cdots,y) $.
    \end{center} Therefore, we have $[x,y^{[p]},z]= (x,y,\cdots,y,z) $ because $\varepsilon(\bar{x},\bar{y})\varepsilon(\bar{y},\bar{x})=1 $. Thus, $T$ is restricted.
\end{proof}

\begin{lem}
Let $D\in \mathrm{Der}(T)_{\alpha}$ with $\alpha\in G_+ $. Then, for any integer $n\geq 1, x,y\in \mathcal{H}(T), z\in T$, we have \begin{center}
        $D^n([x,y,z])= \sum\limits_{i+j+k=n} \binom{n}{ijk}\varepsilon(\alpha,\bar{x})^j\varepsilon(\alpha,\bar{x}+\bar{y})^k[D^i(x),D^j(y),D^k(z)], $
    \end{center} where $\binom{n}{ijk}= \frac{n!}{i!j!k!} $.
\end{lem}

\begin{proof}
   We perform by induction on $n$. The case $n=1$ is clear. Assume that it is true for $n\geq 1$. Let $x,y\in \mathcal{H}(T),z\in T $. We have \begin{align*}
       & D^{n+1}([x,y,z]) = D( \sum\limits_{i+j+k=n} \tbinom{n}{ijk}\varepsilon(\alpha,\bar{x})^j\varepsilon(\alpha,\bar{x}+\bar{y})^k[D^i(x),D^j(y),D^k(z)]) \\ =& \sum\limits_{i+j+k=n} \tbinom{n}{ijk}\varepsilon(\alpha,\bar{x})^j\varepsilon(\alpha,\bar{x}+\bar{y})^k([D^{i+1}x,D^jy,D^kz] +\varepsilon(\alpha,\bar{x})[D^ix,D^{j+1}y,D^kz] \\ & \quad \quad\quad\quad\quad\quad\quad\quad\quad\quad\quad\quad\quad\quad\quad\quad+ \varepsilon(\alpha,\bar{x}+\bar{y})[D^ix,D^jy,D^{k+1}z]).
    \end{align*} As in \cite{H01}, the coefficient of the term $[D^lx,D^my, D^kz ] $  is $\binom{n}{(l-1)mk}+ \binom{n}{l(m-1)k}+ \binom{n}{lm(k-1)} = \binom{n+1}{lmk}$ and the result follows.
\end{proof}

\begin{cor}
    Let $T$ be a Lie color triple system over a field of characteristic $p$. Let $D$ be a derivation of $T$. Then $D^p$ is also a derivation of $T$.
\end{cor}

\begin{proof}
    For $n=p$ the formula above gives \begin{center}
        $D^p([x,y,z])= [D^p(x),y,z]+\varepsilon(\alpha,\bar{x})^p[x,D^p(y),z]+\varepsilon(\alpha,\bar{x}+\bar{y})^p[x,y,D^p(z)] $.
    \end{center}
\end{proof}

\begin{prop}
    Let $T$ be a Lie color triple system. Then, $\mathrm{Der}(T)$ is a restricted Lie color algebra.
\end{prop}

\begin{proof}
    So, $\mathrm{Der}(T) $ is stable under the $p^{th} $ power and then is restricted with the $p$-map $D^{[p]}=D^p $.
\end{proof}

\begin{thm}\label{thm4.9}
    Suppose that $T$ is a Lie color triple system and that $L_S(T)$ is a restricted Lie color algebra with a $p$-map $(\cdot)^{[p]} $. If $T$ is closed under $(\cdot)^{[p]} $, then $T$ is a restricted Lie color triple system.
\end{thm}

\begin{proof}

    Since $T\subset L_S(T)$, Conditions $(1),(2)$ are immediate. Let $\gamma\in G_+ ,x\in L_S(T)_\gamma, y\in \mathcal{H}(L_S(T))$. We have \begin{align*}
         [x^{[p]},y]=(\ad_x)^p(y)& = [x,[x,[\cdots[x,y]]]\cdots]\\ & = -\varepsilon(\bar{x},\bar{y})^p(y,x,\cdots,x) .\end{align*} Therefore for all $x,z\in \mathcal{H}(T)$,$y\in T_\gamma$, we have \begin{align*}
             [x,y^{[p]},z]=[[x,y^{[p]}],z] &=-\varepsilon(\bar{x},\bar{y^{[p]}})[[y^{[p]},x],z]\\ &=\varepsilon(\bar{x},\bar{y})^p\varepsilon(\bar{y},\bar{x})^p(x,y,\cdots,y,z)\\ &=(x,y,\cdots,y,z) . \end{align*} Thus $T$ is a restricted Lie color triple system.
\end{proof}
% Theorem environments

\begin{thm}\label{thm4.10}
    Suppose that $T$ is a Lie color triple system and that $L_S(T)$ is a restricted Lie color algebra. Then $T$ is closed under the $p$-map of $L_S(T)$.
\end{thm}

\begin{proof}
    As in the proof of \cite[Theorem $3.17 $ ]{H01}, we see $T$ is the $(-1)$-eigenspace of the involution $\theta$ of $L_S(T)$ such that $\theta_{\mid T}=-id $ and $\theta_{\mid [T,T]}=id $, therefore we just need to show that $\theta(x^{[p]})=-x^{[p]} $ for all $x\in \mathcal{H}(T_+)$. Since $L_S(T)$ is restricted, we have for all $y\in T$, \begin{center}
        $[x^{[p]},y]= (\ad_x)^p(y) $.
    \end{center} If $y\in \mathcal{H}(T) $, it follows that \begin{center}
        $(\ad_x)^p(y)= [x,[x,[\cdots[x,y]]]]]= -\varepsilon(p\bar{x},\bar{y})(y,x,\cdots,x) $.
    \end{center} Since $p$ is odd, it is the bracket in $L_S(T)$ of two elements in $T$ so $(\ad_x)^p(y)\in [T,T] $. Since $\theta$ is the identity on $[T,T]$, applying $\theta$ to the equality above gives \begin{center}
        $[\theta(x^{[p]}),\theta(y)]=(\ad_x)^p(y) $.
    \end{center} Therefore \begin{center}
        $-[\theta(x^{[p]}),y]= [x^{[p]},y] $.
    \end{center} We showed that $[\theta(x^{[p]}),y]=[-x^{[p]},y] $ for all $y\in T$. We replace $y$ by an element of the form $[y,z]\in [T,T] $ for $y,z\in \mathcal{H}(T) $, we may repeat the calculations above from \begin{center}
        $[x^{[p]},[y,z]]= (\ad_x)^p([y,z]).$ \end{center} Then,  we get \begin{center}
        $(\ad_x)^p([y,z])= -\varepsilon(p\bar{x},\bar{y}+\bar{z})(y,z,x,\cdots,x) $.
    \end{center} So now $(\ad_x)^p([y,z])\in T $. Applying $\theta$, we obtain \begin{center}
        $[\theta(x^{[p]}),[y,z]]= -(\ad_x)^p([y,z]) $.
    \end{center} It follows that $[\theta(x^{[p]}),[y,z]]=[-x^{[p]},[y,z]] $ for all $y,z\in \mathcal{H}(T) $ and thus for all $y,z\in T $. We may conclude that $$[\theta(x^{[p]}),\sum_i[y_i,z_i]]=[-x^{[p]},\sum_i[y_i,z_i]] $$ for all $\sum_i[y_i,z_i] \in [T,T] $ and then  $$[\theta(x^{[p]}),y]=[-x^{[p]},y] $$ for all $y\in L_S(T)$. So $\theta(x^{[p]})+ x^{[p]}\in Z(L_S(T)) $. Since $\theta(\theta(x^{[p]})+x^{[p]})= \theta(x^{[p]})+x^{[p]} $, we can see that $\theta(x^{[p]})+x^{[p]}\in [T,T] $. However, $Z(L_S(T))\cap [T,T]=0 $. Indeed, if $h\in Z(L_S(T))\cap [T,T]$, $[h,c]=0 $ for all $c\in L_S(T)$. In particular, $h= \sum[x_i,y_i]$ and $[\sum_i[x_i,y_i],z]=0 $ for all $z\in T$, so $h=0$ by construction of $L_S(T)$. Hence, the result follows.
\end{proof}

\begin{cor}\label{c4.11}
    Let $T$ be a Lie color triple system. If $L_S(T) $ is a restricted Lie color algebra, then $T$ is a restricted Lie color triple system.
\end{cor}

\begin{proof}
    It follows from  \hyperref[thm4.9]{Theorem ~\ref*{thm4.9}} and \hyperref[thm4.10]{Theorem ~\ref*{thm4.10}}.
\end{proof}

Now, we will show, as in (\cite[Theorem $3.25 $]{H01}), that if the center of $T$ is trivial, a $p$-map on $T$ can be extended as a $p$-map of $L_D(T)$.

\begin{lem}\label{lem4.12}
    Let $L$ be a Lie color algebra. Let $\gamma\in G_+, x\in T_\gamma, y,z\in \mathcal{H}(T) $. Then, we have \begin{center}
        $-\sum\limits_{i=1}^p \varepsilon(\bar{z},\bar{x})^{i-1}(y,x,\cdots,x,z,x,\cdots,x)= \varepsilon(\bar{y},\bar{z}+(p-1)\bar{x})(z,x,\cdots,x,y) $,
    \end{center} where in the left hand side, $x$ occurs $p-1$ times in the bracket, and in the $ith$ summand $z$ appears at the position $i+1 $ starting from $y$.
\end{lem}

\begin{proof} We have \\
     $\sum\limits_{i=1}^n\varepsilon(\bar{z},\bar{x})(y,x,\cdots,x,z,x,\cdots,x)= \sum\limits_{j=0}^{n-1}\varepsilon(\bar{z},\bar{x})^{n-j-1}\ad(-\varepsilon(\bar{y}+\bar{z},\bar{x})x)^j ([ad(-\varepsilon(\bar{y},\bar{x})x)^{n-1-j}(y),z]) $.

     By applying the Leibniz rule on $\ad(-\varepsilon(\bar{y}+\bar{z},\bar{x})x) $, we get \begin{align*}
        & \sum\limits_{j=0}^{n-1}\varepsilon(\bar{z},\bar{x})^{n-j-1}\ad(-\varepsilon(\bar{y}+\bar{z},\bar{x})x)^j ([\ad(-\varepsilon(\bar{y},\bar{x})x)^{n-1-j}(y),z]) \\ & = \sum\limits_{j=0}^{n-1}\sum\limits_{k=0}^j \tbinom{j}{k}\varepsilon(\bar{y},\bar{x})^{n-1-k}\varepsilon(\bar{z},\bar{x})^{n-1}[\ad(-x)^{n-1-k}(y),\ad(-x)^k(z)] \\ & = \sum\limits_{k=0}^nc_{k,n}\varepsilon(\bar{z},\bar{x})^{n-k-1}[\ad(-\varepsilon(\bar{y},\bar{x})x)^{n-1-k}(y), \ad(-\varepsilon(\bar{z},\bar{x})x)^k(z)] \\ & = \sum\limits_{k=1}^n \tbinom{n}{k}\varepsilon(\bar{z},\bar{x})^{n-1-k}[(y,x,\cdots,x),(z,x\cdots,x)],
     \end{align*} where $c_{k,n}= \sum\limits_{j=k}^{n-1}\binom{j}{k} $. For a fixed  $k$, an induction shows that $c_{k,n}= \binom{n}{k+1} $. If $n=p$, it is equal to $[y,(z,x,\cdots,x)]= -\varepsilon(\bar{y},\bar{z}+(p-1)\bar{x})(z,x,\cdots,x,y) $.
\end{proof}

\begin{lem}\label{lem4.13}
    Let $T$ be a restricted Lie color triple system. Assume that $Z(T)=0$. Let $D\in \mathcal{H}(\mathrm{Der}(T))$ viewed as an element of $L_D(T)$ and $x\in \mathcal{H}(T_+)$. Then, we have \begin{center}
        $[x^{[p]},D]= (\ad(x))^p(D) $.
    \end{center}
\end{lem}

\begin{proof}
   We follow the proof of    \cite[Lemma $3.29 $]{H01}. The left hand side of the equation is $-\varepsilon(\bar{x},\bar{D})^pD(x^{[p]}) $ and the right hand side is $-\varepsilon(\bar{x},\bar{D})^p(D(x),x,\dots,x) $. Then we have to show that $D(x^{[p]})=(D(x),x,\dots,x) $. Since the center is trivial, we just need  to show that for all $y,z \in T$ \begin{center}
       $[D(x^{[p]}),y,z]=(D(x),x,\dots,x,y,z) $.
   \end{center} Applying $D$ to $[x^{[p]},y,z] $ for $y,z\in \mathcal{H}(T) $, we obtain \begin{center}
       $D[x^{[p]},y,z]= [D(x^{[p]}),y,z]+\varepsilon(\bar{D},\bar{x})^p[x^{[p]},D(y),z]+\varepsilon(\bar{D},\bar{x})^p\varepsilon(\bar{D},\bar{y})[x^{[p]},y,D(z)] $.
   \end{center} Thus \begin{align*}
       [D(x^{[p]}),y,z] &=D([x^{[p]},y,z]) -\varepsilon(\bar{D},\bar{x})^p[x^{[p]},D(y),z]-\varepsilon(\bar{D},\bar{x})^p\varepsilon(\bar{D},\bar{y})[x^{[p]},y,D(z)] \\ & = -\varepsilon(\bar{x},\bar{y})^pD([y,x^{[p]},z])+ \varepsilon(\bar{x},\bar{y})^p[D(y),x^{[p]},z]+ \varepsilon(\bar{D},\bar{x})^p\varepsilon(\bar{x},\bar{y})^p\varepsilon(\bar{D},\bar{y})[y,x^{[p]},D(z)],
   \end{align*} \begin{center}
       $D([y,x^{[p]},z])= (D(y),x\cdots,x,z) +\varepsilon(\bar{D},\bar{y})\sum\limits_{i=1}^p\varepsilon(\bar{D},\bar{x})^{i-1}(y,x,\cdots,x,D(x),x,\cdots,x,z)+ \varepsilon(\bar{D},\bar{y})\varepsilon(\bar{D},\bar{x})^p(y,x,\cdots,x,D(z)) $ ,
       \end{center} where in the $i^{th}$ term of the sum, $D(x)$ is at position $i+1 $. Finally, by \hyperref[lem4.12]{Lemma ~\ref*{lem4.12}}, we have \begin{align*}
           [D(x^{[p]}),y,z]  =& -\varepsilon(\bar{x},\bar{y})^p(D(y),x,\cdots,x,z)-\varepsilon(\bar{x},\bar{y})^p\varepsilon(\bar{D},\bar{y})\sum\limits_{i=1}^p\varepsilon(\bar{D},\bar{x})^{i-1}(y,x,\cdots,x,D(x),x,\cdots,x,z) \\ &- \varepsilon(\bar{x},\bar{y})^p\varepsilon(\bar{D},\bar{y})\varepsilon(\bar{D},\bar{x})^p(y,x,\cdots,x,D(z))+ \varepsilon(\bar{x},\bar{y})^p(D(y),x,\cdots,x,z) \\ &+ \varepsilon(\bar{x},\bar{y})^p\varepsilon(\bar{D},\bar{y})\varepsilon(\bar{D},\bar{x})^p(y,x,\cdots,x,D(z)) \\ =&  -\varepsilon(\bar{x},\bar{y})^p\varepsilon(\bar{D},\bar{y})\sum\limits_{i=1}^p\varepsilon(\bar{D},\bar{x})^{i-1}  (y,x,\cdots,x,D(x),x,\cdots,x,z) \\ =& (D(x),x,\cdots,x,y,z),
        \end{align*} for all $y,z\in \mathcal{H}(T) $.
   
\end{proof}

\begin{lem}\label{lem4.14}
    Let $x\in \mathcal{H}(T_+)$, $y\in T$ viewed as elements of $L_D(T) $. Then, we have \begin{center}
        $[x^{[p]},y]= (\ad_x)^p(y) $.
    \end{center}
\end{lem}

\begin{proof}
    Let $ D=(\ad_x)^p$. By the Leibniz rule, $D$ is a derivation of $L_D(T) $. So for all $y,z \in \mathcal{H}(T)$, we have \begin{align*}
        [D(y),z] &=D([y,z])-\varepsilon(\bar{D},\bar{y})[y,D(z)] \\ & = -\varepsilon(\bar{x},\bar{y}+\bar{z})^p(y,z,x\dots,x)-\varepsilon(\bar{y},\bar{z})\varepsilon(\bar{x},\bar{z})^p(z,x,\dots,x,y).
    \end{align*} The left hand side is $[(ad_x)^p(y),z] $. The right hand side is \begin{align*}  - & \varepsilon(\bar{x},\bar{y}+\bar{z})^p(y,z,x,\dots,x)-\varepsilon(\bar{y},\bar{z})\varepsilon(\bar{x},\bar{z})^p(z,x,\dots,x,y) \\ & =  -\varepsilon(\bar{x},\bar{y}+\bar{z})^p[y,z,x^{[p]}]-\varepsilon(\bar{y},\bar{z})\varepsilon(\bar{x},\bar{z})^p[z,x^{[p]},y] =[x^{[p]},y,z] .\end{align*} So, $[(\ad_x)^p(y),z]= [[x^{[p]},y],z] $ and since $p$ is odd, $(\ad_x)^p(y) $ and $[x^{[p]},y] $ are elements of $\mathrm{Der}(T)$. Then $(\ad_x)^p(y)= [x^{[p]},y] $.
\end{proof}
Now, we prove the following theorem, with a proof similar to (\cite[Theorem $3.25 $]{H01}).

\begin{thm}\label{thm4.15}
    Let $T$ be a restricted Lie color triple system with $Z(T)=0$. Then $L_D(T)$ is a restricted Lie color algebra, with a restricted structure extending that of $T$.
\end{thm}

\begin{proof}
    For $D\in \mathcal{H}(\mathrm{Der}(T)_+) $, we define $D^{[p]}=D^p $ and for $x\in \mathcal{H}(T_+)\subset \mathcal{H}(L_D(T)_+) $, we define $x^{[p]} $, where $(\cdot)^{[p]} $ is the $p$-map of $T$. We will use the color case Jacobson's Theorem. Let $\gamma\in G_+ $. We just need to  show that $[x^{[p]},y]=(\ad_x)^p(y) $ for all $y\in L_D(T) $ and for all $x\in B$, where $B$ is a basis of $L_D(T)_\gamma$. For $B$, we choose a basis which is the concatenation of a basis $B_T$ of $T_\gamma$ and a basis $B_D$ of $\mathrm{Der}(T)_\gamma $. Let $y= D'+b\in L_D(T)$, $D'\in \mathrm{Der}(T)$, $b\in T$. Assume that $ x=D\in B_D $, then $[D,b]=D(b) $, so $(\ad_D)^p(b)=D^p(b) $. We have \begin{align*}
        [D^p,D'+b] &= [D^p,D'] +D^p(b) \\ & = (\ad_D)^p(D')+ (\ad_D)^p(b)\\ & = (\ad_D)^p(D'+b).
    \end{align*} If $x=a\in B_T $, according to \hyperref[lem4.13]{Lemma ~\ref*{lem4.13}} and \hyperref[lem4.14]{Lemma~\ref*{lem4.14} }, we have \begin{align*} [a^{[p]},D'+b] &= (\ad_a)^p(D')+(\ad_a)^p(b)\\ &= (\ad_a)^p(D'+b) . \end{align*} So $L_D(T) $ has a restricted structure.
\end{proof}

We aim to  prove an analog of  Jacobson's Theorem for Lie color triple systems. It is shown in \cite{HP02} that the Jacobson's Theorem is true for Lie triple systems. Here we will prove that it is true for Lie color triple systems, using a similar proof as  \cite[Theorem $2.3 $]{SF88}.

\begin{thm}\label{thm4.16}
    Let $T$ be a Lie color triple system over a field $\mathbf{K}$ of characteristic $p>3 $. Suppose that for all $\gamma\in G_+ $ there exists a basis $(u_{i,\gamma})_{i\in J}$  of $T_\gamma$ with the property that for every $i$ there exists a $v_{i,\gamma}\in T_{p\gamma} $ such that \begin{center}
        $[a,v_{i,\gamma},c]= (a,u_{i,\gamma},\dots,u_{i,\gamma},c) $ for all $a,c\in T.$
    \end{center} Then, there is a unique restricted structure on $T$ such that $u_{i,\gamma}^{[p]}=v_{i,\gamma} $ for all $i$.
    
\end{thm}
%\begin{proof}Similar proof that for the Lie color algebra, using the standard embedding of a Lie color triple system.
%\end{proof}
For a Lie color triple system $T$ and a subset $S\subset T$, we define the centralizer of $S$ in $T$ as \begin{center}
    $C_T(S)= \left\{x\in T\mid [x,a,b]=0 , \forall a,b\in S \right\} .$
\end{center}

\begin{prop}\label{p4.17}
     Let $S$ be a color subsystem of a restricted Lie color triple system $T$ and $(\cdot)^{[p]_1}:S_\gamma\to S_{p\gamma} $ a map for all $\gamma\in G_+ $. Then $(\cdot)^{[p]_1} $ is a $p$-map if and only if there exists a $p$-semi-linear map of degree $p$ denoted by $f_\gamma:S_\gamma\to C_T(S)$ such that $(\cdot)^{[p]_1}=(\cdot)^{[p]}+f $ for all $\gamma\in G_+ $.
\end{prop}

\begin{proof}
    The same proof as in the color algebra case. Assume that $(\cdot)^{[p]_1} $ is a $p$-map and  let $\gamma\in G_+ $. Then for all $x\in L_\gamma$, we have $x^{[p]}-x^{[p]_1}\in C_T(S) $. Then  verify that $f(x):= x^{[p]}-x^{[p]_1} $ is a $p$-semi-linear map. Conversely, assume that there exists a $p$-semi-linear map $f: S_\gamma\to C_T(S)$ such that $(\cdot)^{[p]_1}= (\cdot)^{[p]}+ f $. Then, we have \begin{align*}
        (x+y)^{[p]_1} &= (x+y)^{[p]}+ f(x+y) \\ &= x^{[p]}+y^{[p]} + \sum\limits_{i=1}^p s_i(x,y) + f(x) +f(y) \\ & = x^{[p]_1}+y^{[p]_1} + \sum\limits_{i=1}^ps_i(x,y). 
    \end{align*} The other conditions of Definition $4.1$ are straightforward. 
\end{proof}

\begin{proof}[Proof of the theorem]
    The Lie color triple system $T$ is embedded in its standard embedding $L_S(T)$. From the PBW-Theorem, $L_S(T) $ is embedded in its universal algebra $U(L_S(T)) $. Let $\gamma\in G_+ $. For all $a\in \mathcal{H}(T),$  $ c\in T $, we have: \begin{align*}
        -\varepsilon(\bar{a},\bar{v_{i,\gamma}})[v_{i,\gamma},a,c]= [a,v_{i,\gamma},c]= & (a,u_{i,\gamma},\cdots,u_{i,\gamma},c) \\ =& [[\cdots [a,u_{i,\gamma}],u_{i,\gamma}],\cdots ,u_{i,\gamma}],c ] \\ =& (-1)^p\varepsilon(\bar{a},\bar{u_{i,\gamma}}) \varepsilon(\bar{a}+\bar{u_{i,\gamma}},\bar{u_{i,\gamma}})\cdots\varepsilon(\bar{a}+(p-1)\bar{u_{i,\gamma}})[\ad_{u_{i,\gamma}}^p(a),c] \\ =&  -\varepsilon(\bar{a},\bar{u_{i,\gamma}})^p[\ad_{u_{i,\gamma}}^p(a),c] .\end{align*} Thus it follows that $\ad_{v_{i,\gamma}}(a)= \ad_{u_{i,\gamma}}^p(a)$ for all $a\in T$, by definition of $L_S(T)$ and then we have $ \ad_{v_{i,\gamma}}= \ad_{u_{i,\gamma}}^p$. In $U(L_S(T)) $, we have $ \ad_{u_{i,\gamma}}^p=\ad_{u_{i,\gamma}^p} $. Then it follows that $v_{i,\gamma} - u_{i,\gamma}^p  \in C_{U(L_S(T))}(L_S(T)) $. Since the  Lie color triple system $U(L_S(T)) $ is restricted, $T$ is a subsystem of $U(L_S(T)) $ and $ C_{U(L_S(T))}(L_S(T))\subset C_{U(L_S(T))}(T) $. We define a map $f_\gamma: T_\gamma\to C_{U(L_S(T))}(T) $, \begin{center}
        $f_\gamma(\sum\alpha_iu_{i,\gamma}):= \sum\alpha^p(v_{i,\gamma}-u_{i,\gamma}^p), \alpha_i\in \mathbf{K} $. \end{center} Denote by \begin{center}
        $ V = \left\{x\in T_\gamma , x^p+f_\gamma(x) \in T_{p\gamma}\right\} $. \end{center} Since $(\alpha x+y)^p+f_\gamma(\alpha x+y) = \alpha^px^p+y^p+\sum\limits_{i=1}^ps_i(\alpha x,y) + \alpha^pf_\gamma(x)+ f_\gamma(y) \in T\cap U(L_S(T))_{p\gamma}= T_{p\gamma} $, $V$ is a vector space. The elements $(u_{i,\gamma})_{j\in J}$ lie in $V$ and so $V=T_\gamma$. Then, with  \hyperref[p4.17]{Proposition ~\ref*{p4.17}},  we have $ (\cdot)^{[p]_1}:= x\mapsto x^p+f_\gamma(x) $ is a $p$-map on $T$ and it follows that $T$ is restricted.
\end{proof}

\section{Representations}\label{s5}

In this section, we define  representations of  Lie color triple systems and  restricted representations of  restricted Lie color algebras and  restricted Lie color triple systems. We  show, using Jacobson's Theorem, that  the semi-direct product of a restricted Lie color algebra and its restricted representation has a  restricted Lie color triple algebra structure, as well as   a restricted Lie color triple system and its restricted representation is equipped with  a restricted Lie color triple system structure. 

\subsection{Representations of Lie color algebras and Lie color triple systems}

\begin{defi}[\cite{F01}]
    Let $L$ be a Lie color algebra. A representation of $L$ is a morphism of Lie algebra $\varphi:L\to \mathrm{End}(V) $, where $V$ is a $G$-graded vector space and $\mathrm{End}(V)$ is a Lie color algebra with the bracket $[f,g]=fg-\varepsilon(\bar{f},\bar{g})gf $ for $f,g$ two homogeneous elements of $\mathrm{End}(V) $. 
\end{defi}

%In "1-parameter formal deformations and abelian extensions of Lie color triple systems" we have (it seems to be inspired by the paper of Wolfgang Bertram and Didry Manon)
We consider a definition of a representation of a Lie color triple system with respect to an endomorphism given in \cite[Definition $2.4 $]{LM22}.
\begin{defi}\label{d5.2}
    Let $T$ be a Lie color triple system, $V$ be a $G$-graded $\mathbf{K}$-vector space and $\theta:T\times T\to \mathrm{End}(V) $ be a bilinear map. A representation of $T$ is a pair $(V,\theta)$ satisfying, for all $x,y,z,t\in \mathcal{H}(T)$, 
           \begin{align}
          & \theta(x,y)(V_\alpha)\subset V_{\alpha+\bar{x}+\bar{y}} ,\label{rep1}\\
           % \begin{multline*}
            &   \varepsilon(\bar{x}+\bar{y},\bar{z}+\bar{t})\theta(z,t)\theta(x,y)-\varepsilon(\bar{x},\bar{y})\varepsilon(\bar{x}+\bar{z},\bar{t})\theta(y,t)\theta(x,z)-\theta(x,[y,z,t])\nonumber\\& \quad + \varepsilon(\bar{x},\bar{y}+\bar{z})D(y,z)\theta(x,t)=0, %\end{multline*}
               \label{rep2} \\ &
          \varepsilon(\bar{x}+\bar{y},\bar{z}+\bar{t})\theta(z,t)D(x,y)-D(x,y)\theta(z,t)+\theta([x,y,z],t)+ \varepsilon(\bar{x}+\bar{y},\bar{z})\theta(z,[x,y,t])=0 ,\label{rep3}
          \end{align} where $D(x,y)=\varepsilon(\bar{x},\bar{y})\theta(y,x)-\theta(x,y) $.
\end{defi}

\begin{ex}
     $(\theta,T)$ with $\theta(x,y)(z)=\varepsilon(\bar{x}+\bar{y},\bar{z})[z,x,y] $ is a representation, called the adjoint representation.
\end{ex}

\begin{ex}
    We provide here all the $1$ and $2$ dimensional representations   of the Lie color triple system $L=L_{(1,1,0)}\oplus L_{(1,0,1)} \oplus L_{(0,1,1)} $ defined in Example \ref{Example1}. 
    
    Let $V= \mathbf{K}v$ be a 1-dimensional $\Gamma $-graded vector space. Let $\theta:T\times T\to \mathrm{End}(V) $ be a bilinear map. No matter what the degree of $v$, it follows from  Condition \eqref{rep1}, that for $i\neq j \in \left\{1,2,3 \right\} $, $\theta(e_i,e_i)(v)= \alpha_iv  $  and $\theta(e_i,e_j)(v)=0 $. Then, by taking $x=e_3,y=e_1,z=e_3,t=e_1 $ in  Condition \eqref{rep3}, we obtain that $\alpha_{1}=\alpha_3 $. Then, with the quadruplet $x=e_3,y=e_2,z=e_2,t=e_3 $ we have that $\alpha_{3}=\alpha_2 $. By taking $x=e_1,y=e_2,z=e_1$ and $t=e_2$ we obtain that $\alpha^2-\alpha=0 $. It follows that Conditions \eqref{rep2} and \eqref{rep3} are automatically verified. Thus, for $\alpha=0,1$ $\in \mathbf{K}$, the bilinear map $\theta:T\times T\to \mathrm{End}(V) $ defined with respect to  the basis by \begin{center}
          $\theta(e_i,e_i)(v)= \alpha v  $   and $\theta(e_i,e_j)(v)=0 $  for all $i\neq j \in \left\{1,2,3 \right\}$
    \end{center} is a representation of $L$. So there is only one non-trivial one dimensional representation of $L$.
    
    Let $V=\mathbf{K}v_1 \oplus \mathbf{K}v_2 $ be a $2$-dimensional $\Gamma $-graded vector space. Assume that $v_1,v_2$ are not of same degree. We have to deal with cases where $(\bar{v_1},\bar{v_2})\in \left\{(0,1),(0,2),(0,3),(1,2),(1,3),(2,3) \right\} $. If $(\bar{v_1},\bar{v_2})=(1,2) $, then Condition \eqref{rep2} gives that \begin{center}
        $\theta(e_i,e_i)(v_j)= \alpha_{i}^jv_j $, $\theta(e_3,e_i)=0 $, $\theta(e_i,e_3)=0 $, 

        $\theta(e_1,e_2)(v_1)= \alpha_{1,2}^1v_2 $, $\theta(e_1,e_2)(v_2)= \alpha_{1,2}^2v_1 $, 

        $ \theta(e_2,e_1)(v_1)= \alpha_{2,1}^1v_2 $, $\theta(e_2,e_1)(v_2)= \alpha_{2,1}^2v_1 $. 
    \end{center} With  Condition \eqref{rep3}, we obtain by taking $x=e_3,y=e_1 ,z=e_2 ,t=_3 $ that $\alpha_{2,1}^1=\alpha_{2,1}^2=0 $. By taking $x=e_3,y=e_1,z=e_3,t=e_1 $, we obtain that $\alpha_{1}^1=\alpha_3^1 $ and $\alpha_1^2=\alpha_3^2 $. Then with the quadruplet $x=e_3,y=e_2,z=e_1, t=e_3 $, we have that $\alpha_{1,2}^1=\alpha_{1,2}^2= 0 $ and with $x=e_3,y=e_2,z=e_2,t=e_3 $, we have that $\alpha_{3}^1=\alpha_2^1 $ and $\alpha_3^2=\alpha_2^2 $. By taking $x=e_1,y=e_2,z=e_1$ and $t=e_2$, we obtain that $\alpha^2-\alpha=0 $ and $\beta^2-\beta=0 $. So if $(\theta,V)$ is a representation, then $\theta(e_i,e_j)=0 $ if $i\neq j$ and there exist two scalars $\alpha=0,1$ and $\beta=0,1 $ such that $\theta(e_i,e_i)(v_1)=\alpha v_1 $ and $\theta(e_i,e_i)(v_2)=\beta v_2 $. The same result is obtained when considering another case of pair $(v_1,v_2)$ of different degrees.
    
    If $V= \mathbf{K}v\oplus \mathbf{K}w$, where  $v,w$ are generators of the same degree $i$, then $\theta(e_i,e_i)(v)= \alpha_iv+\beta_iw $, $\theta(e_i,e_i)(w)= \gamma_iv+ \delta_iw $ and $\theta(e_i,e_j)=0 $ for $i\neq j$. Using  Condition \eqref{rep3}, we have that $\theta(e_i,e_i)=\theta(e_j,e_j) $. So there exists four scalar, $\alpha,\beta,\gamma,\delta $ such that $\theta(e_i,e_i)(v) = \alpha v+\beta w $ and $\theta(e_i,e_i)(w)= \gamma v+ \delta w $. By taking $x=e_1,y=e_2,z=e_1,t=e_2 $ in Condition \eqref{rep3}, we get $\theta(e_1,e_1)^2-\theta(e_1,e_1)=0. $ Thus we have this system of equations \begin{center}
        $\left\{\begin{matrix}
\alpha^2+\beta\gamma=\alpha \\ \alpha\beta+\beta\delta=\beta
 \\\alpha\gamma+\delta\gamma=\gamma
 \\ \gamma\beta+\delta^2=\delta
\end{matrix}\right.$
   \end{center} If we denote by \begin{center}
        $A=\begin{pmatrix}
\alpha  & \gamma  \\
\beta  & \delta   \\
\end{pmatrix} $
    \end{center} the matrix of $\theta(e_i,e_i) $, then $A$ is an idempotent of $\mathcal{M}_2(\mathbf{K}). $ The matrix idempotent in $\mathcal{M}_2(\mathbf{K}) $ are $0, \mathrm{I}_2$ and the projectors of rank $1$, which correspond to the case $\alpha+\delta=1 $, $\beta\gamma= \alpha(1-\alpha). $ Thus, the non-trivial representations of $L$ are the bilinear maps $\theta:T\times T\to \mathrm{End}(T) $ such that $\theta(e_i,e_j)=0 $ for $i\neq j$ and $\theta(e_i,e_i)=\mathrm{Id}_V $ or $\theta^2
    (e_i,e_i)=\theta(e_i,e_i) $ and $\theta(e_i,e_i) $ is a projector of rank $1$ of $V$. 
\end{ex}

\begin{prop}
    Let $L$ be a Lie color algebra and $\varphi: L\to \mathrm{End}(V)$ be a representation of $L$. Then, $\varphi$ induces a representation of $L_{\mathrm{trip}} $ defined by $\theta(x,y)=\varepsilon(\bar{x},\bar{y})\varphi(y)\varphi(x) $.
\end{prop}

\begin{proof}
    Let us show that $\theta$ satisfies the conditions of  \hyperref[d5.2]{Definition ~\ref*{d5.2}}.

    $1$. Let $v\in V_{\alpha} $ and $x,y\in \mathcal{H}(L)$, $\varphi(x)(v)\in V_{\alpha+\bar{x}} $.  Then $\varphi(y)(\varphi(x)(v))\in V_{\alpha+\bar{x}+\bar{y}} $. 

    $2.$ Let $x,y,z,t\in \mathcal{H}(T) $. We have \begin{align*}   
    &\varepsilon(\bar{x}+\bar{y},\bar{z}+\bar{t})\theta(z,t)\theta(x,y)-\varepsilon(\bar{x},\bar{y})\varepsilon(\bar{x}+\bar{z},\bar{t})\theta(y,t)\theta(x,z) \\ &-\theta(x,[y,z,t])+ \varepsilon(\bar{x},\bar{y}+\bar{z})D(y,z)\theta(x,t) \\ 
&=\varepsilon(\bar{x}+\bar{y},\bar{z}+\bar{t})\varepsilon(\bar{z},\bar{t})\varepsilon(\bar{x},\bar{y})\varphi(t)\varphi(z)\varphi(y)\varphi(x)\\
&-\varepsilon(\bar{x},\bar{y})\varepsilon(\bar{x}+\bar{z},\bar{t})\varepsilon(\bar{y},\bar{t})\varepsilon(\bar{x},\bar{z})\varphi(t)\varphi(y)\varphi(z)\varphi(x) - \varepsilon(\bar{x},\bar{y}+\bar{z}+\bar{t})\varphi([[y,z],t])\varphi(x) \\ &+ \varepsilon(\bar{x},\bar{y}+\bar{z})(\varepsilon(\bar{y},\bar{z})\varepsilon(\bar{z},\bar{y})\varphi(y)\varphi(z)- \varepsilon(\bar{y},\bar{z})\varphi(z)\varphi(y))\varepsilon(\bar{x},\bar{t})\varphi(t)\varphi(x)\\ &=\varepsilon(\bar{x}+\bar{y},\bar{z}+\bar{t})\varepsilon(\bar{z},\bar{t})\varepsilon(\bar{x},\bar{y})\varphi(t)\varphi(z)\varphi(y)\varphi(x) \\ &- \varepsilon(\bar{x},\bar{y})\varepsilon(\bar{x}+\bar{z},\bar{t})\varepsilon(\bar{y},\bar{t})\varepsilon(\bar{x},\bar{z})\varphi(t)\varphi(y)\varphi(z)\varphi(x) \\ &- \varepsilon(\bar{x},\bar{y}+\bar{z}+\bar{t})(\varphi(y)\varphi(z)\varphi(t)\varphi(x)- \varepsilon(\bar{y},\bar{z})\varphi(z)\varphi(y)\varphi(t)\varphi(x)\\ &- \varepsilon(\bar{y}+\bar{z},\bar{t})\varphi(t)\varphi(y)\varphi(z)\varphi(x)+ \\  
&\varepsilon(\bar{y}+\bar{z},\bar{t})\varepsilon(\bar{t},\bar{z})\varphi(t)\varphi(z)\varphi(y)\varphi(x))+ \varepsilon(\bar{x},\bar{y}+\bar{z})\varepsilon(\bar{x},\bar{t})\varphi(y)\varphi(z)\varphi(t)\varphi(x) \\ & -\varepsilon(\bar{x},\bar{y}+\bar{z})\varepsilon(\bar{y},\bar{z})\varepsilon(\bar{x},\bar{t})\varphi(z)\varphi(y)\varphi(t)\varphi(x) \\
&= (\varepsilon(\bar{x}+\bar{y},\bar{z}+\bar{t})\varepsilon(\bar{z},\bar{t})\varepsilon(\bar{x},\bar{y}) - \varepsilon(\bar{x},\bar{y}+\bar{z}+\bar{t})\varepsilon(\bar{y}+\bar{z},\bar{t})\varepsilon(\bar{y},\bar{z}))\varphi(t)\varphi(z)\varphi(y)\varphi(x) \\
&+ (\varepsilon(\bar{x},\bar{y}+\bar{z}+\bar{t})\varepsilon(\bar{y}+\bar{z},\bar{t})- \varepsilon(\bar{x},\bar{y})\varepsilon(\bar{x}+\bar{z},\bar{t})\varepsilon(\bar{y},\bar{t})\varepsilon(\bar{x},\bar{z}))\varphi(t)\varphi(y)\varphi(z)\varphi(x) \\
    &+ (\varepsilon(\bar{x},\bar{y}+\bar{z})\varepsilon(\bar{x},\bar{t})- \varepsilon(\bar{x},\bar{y}+\bar{z}+\bar{t}))\varphi(y)\varphi(z)\varphi(t)\varphi(x) \\
    &+ (\varepsilon(\bar{x},\bar{y}+\bar{z}+\bar{t})\varepsilon(\bar{y},\bar{z})- \varepsilon(\bar{x},\bar{y}+\bar{z})\varepsilon(\bar{y},\bar{z})\varepsilon(\bar{x},\bar{t}))\varphi(z)\varphi(y)\varphi(t)\varphi(x) \\
    &= 0. 
    \end{align*}

    $3.$ Let $x,y,z,t\in \mathcal{H}(T) $. We have \begin{align*}
        &\varepsilon(\bar{x}+\bar{y},\bar{z}+\bar{t})\theta(z,t)D(x,y)-D(x,y)\theta(z,t)+\theta([x,y,z],t)+ \varepsilon(\bar{x}+\bar{y},\bar{z})\theta(z,[x,y,t])\\ & = \varepsilon(\bar{x}+\bar{y},\bar{z}+\bar{t})\varepsilon(\bar{z},\bar{t})\varphi(t)\varphi(z)(\varepsilon(\bar{x},\bar{y})\varepsilon(\bar{y},\bar{x})\varphi(x)\varphi(y)-\varepsilon(\bar{x},\bar{y})\varphi(y)\varphi(x)) \\ & - (\varepsilon(\bar{x},\bar{y})\varepsilon(\bar{y},\bar{x})\varphi(x)\varphi(y)- \varepsilon(\bar{x},\bar{y})\varphi(y)\varphi(x))\varepsilon(\bar{z},\bar{t})\varphi(t)\varphi(z) + \varepsilon(\bar{x}+\bar{y}+\bar{z},\bar{t})\varphi(t)\varphi(x)\varphi(y)\varphi(z) \\ & - \varepsilon(\bar{x}+\bar{y}+\bar{z},\bar{t})\varepsilon(\bar{x},\bar{y})\varphi(t)\varphi(y)\varphi(x)\varphi(z) - \varepsilon(\bar{x}+\bar{y}+\bar{z},\bar{t})\varepsilon(\bar{x}+\bar{y},\bar{z})\varphi(t)\varphi(z)\varphi(x)\varphi(y) \\ & + \varepsilon(\bar{x}+\bar{y}+\bar{z},\bar{t})\varepsilon(\bar{x}+\bar{y},\bar{z})\varepsilon(\bar{x},\bar{y})\varphi(t)\varphi(z)\varphi(y)\varphi(x) + \varepsilon(\bar{x}+\bar{y},\bar{z})\varepsilon(\bar{z},\bar{x} \\ & +\bar{y}+\bar{t})\varphi(x)\varphi(y)\varphi(t)\varphi(z) -\varepsilon(\bar{x}+\bar{y},\bar{z})\varepsilon(\bar{z},\bar{x}+\bar{y}+\bar{t})\varepsilon(\bar{x},\bar{y}) \varphi(y)\varphi(x)\varphi(t)\varphi(z) \\ & - \varepsilon(\bar{x}+\bar{y},\bar{z})\varepsilon(\bar{z},\bar{x} +\bar{y}+\bar{t})\varepsilon(\bar{x}+\bar{y},\bar{t})\varphi(t)\varphi(x)\varphi(y)\varphi(z) \\ & \varepsilon(\bar{x}+\bar{y},\bar{z})\varepsilon(\bar{z},\bar{x}+\bar{y}+\bar{t})\varepsilon(\bar{x}+\bar{y},\bar{t})\varepsilon(\bar{x},\bar{y})\varphi(t)\varphi(y)\varphi(x)\varphi(z) \\ & = (\varepsilon(\bar{x}+\bar{y},\bar{z}+\bar{t})\varepsilon(\bar{z},\bar{t})-\varepsilon(\bar{x}+\bar{y}+\bar{z},\bar{t})\varepsilon(\bar{x}+\bar{y},\bar{z}))\varphi(t)\varphi(z)\varphi(x)\varphi(y) \\ & + (-\varepsilon(\bar{x}+\bar{y},\bar{z}+\bar{t})\varepsilon(\bar{z},\bar{t})\varepsilon(\bar{x},\bar{y})+ \varepsilon(\bar{x}+\bar{y}+\bar{z},\bar{t})\varepsilon(\bar{x}+\bar{y},\bar{z})\varepsilon(\bar{x},\bar{y}))\varphi(t)\varphi(z)\varphi(y)\varphi(x) \\ & + (\varepsilon(\bar{x}+\bar{y},\bar{z})\varepsilon(\bar{z},\bar{x}+\bar{y}+\bar{t})-\varepsilon(\bar{z},\bar{t}))\varphi(x)\varphi(y)\varphi(t)\varphi(z) \\ & + (\varepsilon(\bar{x},\bar{y})\varepsilon(\bar{z},\bar{t})- \varepsilon(\bar{x}+\bar{y},\bar{z})\varepsilon(\bar{z},\bar{x}+\bar{y}+\bar{t})\varepsilon(\bar{x},\bar{y}))\varphi(y)\varphi(x)\varphi(t)\varphi(z) \\ & 
 +(\varepsilon(\bar{x}+\bar{y}+\bar{z},\bar{t})- \varepsilon(\bar{x}+\bar{y},\bar{z})\varepsilon(\bar{z},\bar{x}+\bar{y}+\bar{t})\varepsilon(\bar{x}+\bar{y},\bar{t}))\varphi(t)\varphi(x)\varphi(y)\varphi(z) \\ & + (\varepsilon(\bar{x}+\bar{y},\bar{z})\varepsilon(\bar{z},\bar{x}+\bar{y}+\bar{t})\varepsilon(\bar{x}+\bar{y},\bar{t})\varepsilon(\bar{x},\bar{y})- \varepsilon(\bar{x}+\bar{y}+\bar{z},\bar{t})\varepsilon(\bar{x},\bar{y}))\varphi(t)\varphi(y)\varphi(x)\varphi(z) \\ & =0.
        \end{align*} 
\end{proof}

\begin{prop}
    Let $(\theta,V)$ be a representation of a Lie color triple system $T$. Consider the operation $[\cdot,\cdot,\cdot]_V: (T\oplus V)\times (T\oplus V)\times (T\oplus V)\to T\oplus V  $ given, on homogeneous elements, by \begin{center}
        $[(x,a),(y,b),(z,c)]_V=([x,y,z], \varepsilon(\bar{x},\bar{y}+\bar{z})\theta(y,z)(a)-\varepsilon(\bar{y},\bar{z})\theta(x,z)(b)+ D(x,y)(c)) $.
    \end{center} Then, $T\oplus V$ is a Lie color triple system.
\end{prop}

Here, a homogeneous element of $T\oplus V $ is of the form $(x,v) $, where $x\in T_\gamma$, $v\in V_\gamma$ for $\gamma\in G $ and  the graduation on $T\oplus V $ is that $(x,v) $ is of degree $\gamma \in G$ if $x,v$ are of degree $\gamma $.
\begin{proof}
    We will show that $T\oplus V$, endowed with the bracket $[\cdot,\cdot,\cdot]_V $ satisfies the conditions being a Lie color triple system in \hyperref[d2.8]{Definition~\ref*{d2.8}}. Let $(x,a),(y,b),(z,c),(u,d),(v,e) \in \mathcal{H}(T\oplus V) $.

    For Condition \eqref{lts1}, we have $[x,y,z]\subset T_{\bar{x}+\bar{y}+\bar{z}} $ and by definition of $\theta$, we have \\ $\theta(y,z)(a), \theta(x,z)(b), D(x,y)(c) \in V_{\bar{x}+\bar{y}+\bar{z}} $. Thus  $[(x,a),(y,b),(z,c)]_V\in (T\oplus V)_{\bar{x}+\bar{y}+\bar{z}} $.

    For Condition \eqref{lts2},  we have \begin{align*}
        & [(x,a),(y,b),(z,c)]_V  = ([x,y,z],\varepsilon(\bar{x},\bar{y}+\bar{z})\theta(y,z)(a)- \varepsilon(\bar{y},\bar{z})\theta(x,z)(b)+ D(x,y)(c)) \\ & = -\varepsilon(\bar{x},\bar{y})(-\varepsilon(\bar{y},\bar{x})[x,y,z], -\varepsilon(\bar{x},\bar{z})\theta(y,z)(a)+\varepsilon(\bar{y},\bar{x}+\bar{z})\theta(x,z)(b)+D(y,x)(c)) \\ & = -\varepsilon(\bar{x},\bar{y})([y,x,z],\varepsilon(\bar{y},\bar{x}+\bar{z})\theta(x,z)(b)-\varepsilon(\bar{x},\bar{z})\theta(y,z)(a)+D(y,x)(c)) \\ & = -\varepsilon(\bar{x},\bar{y})[(y,b),(x,a),(z,c)]_V.
    \end{align*}

     For Condition \eqref{lts3}, \begin{align*}
        & \varepsilon(\bar{z},\bar{x})[(x,a),(y,b),(z,c)]_V + \varepsilon(\bar{x},\bar{y})[(y,b),(z,c),(x,a)]_V + \varepsilon(\bar{y},\bar{z})[(z,c),(x,a),(y,b)]_V \\ & =\varepsilon(\bar{z},\bar{x})([x,y,z],\varepsilon(\bar{x},\bar{y}+\bar{z})\theta(y,z)(a)-\varepsilon(\bar{y},\bar{z})\theta(x,z)(b) + D(x,y)(c)) \\ & + \varepsilon(\bar{x},\bar{y})([y,z,x], \varepsilon(\bar{y},\bar{x}+\bar{z})\theta(z,x)(b)- \varepsilon(\bar{z},\bar{x})\theta(y,x)(c)+ D(y,z)(a)) \\ & + \varepsilon(\bar{y},\bar{z})([z,x,y], \varepsilon(\bar{z},\bar{x}+\bar{y})\theta(x,y)(c)-\varepsilon(\bar{x},\bar{y})\theta(z,y)(a) + D(z,x)(b)) \\ & = \varepsilon(\bar{z},\bar{x})([x,y,z],\varepsilon(\bar{x},\bar{y}+\bar{z})\theta(y,z)(a)-\varepsilon(\bar{y},\bar{z})\theta(x,z)(b) + \varepsilon(\bar{x},\bar{y})\theta(y,x)(c)-\theta(x,y)(c)) \\ & + \varepsilon(\bar{x},\bar{y})([y,z,x], \varepsilon(\bar{y},\bar{x}+\bar{z})\theta(z,x)(b)- \varepsilon(\bar{z},\bar{x})\theta(y,x)(c)+ \varepsilon(\bar{y},\bar{z})\theta(z,y)(a)-\theta(y,z)(a)) \\ & + \varepsilon(\bar{y},\bar{z})([z,x,y], \varepsilon(\bar{z},\bar{x}+\bar{y})\theta(x,y)(c)-\varepsilon(\bar{x},\bar{y})\theta(z,y)(a) + \varepsilon(\bar{z},\bar{x})\theta(x,z)(b)-\theta(z,x)(b)) \\ & =(0,0).
    \end{align*}

   For Condition \eqref{lts4}, on the one hand we have  \begin{align*}
        & [(u,d),(v,e),[(x,a),(y,b),(z,c)]]_V \\ & = [(u,d),(v,e),([x,y,z],\varepsilon(\bar{x},\bar{y}+\bar{z})\theta(y,z)(a)-\varepsilon(\bar{y},\bar{z})\theta(x,z)(b)+ D(x,y)(c)) ]_V \\ & =([u,v,[x,y,z]],\varepsilon(\bar{u},\bar{v}+\bar{x}+\bar{y}+\bar{z})\theta(v,[x,y,z])(d) - \varepsilon(\bar{v},\bar{x}+\bar{y}+\bar{z})\theta(u,[x,y,z])(e) \\ &+ D(u,v)(\varepsilon(\bar{x},\bar{y}+\bar{z})\theta(y,z)(a)-\varepsilon(\bar{y},\bar{z})\theta(x,z)(b)+ D(x,y)(c))).
    \end{align*} On the other hand, we have \begin{align*}
       & [[(u,d),(v,e),(x,a)]_V,(y,b),(z,c)]_V \\ &=([[u,v,x],y,z],\varepsilon(\bar{u}+\bar{v}+\bar{x},\bar{y}+\bar{z})\theta(y,z)(\varepsilon(\bar{u},\bar{v}+\bar{x})\theta(v,x)(d)-\varepsilon(\bar{v},\bar{x})\theta(u,x)(e)+D(u,v)(a))\\ &-\varepsilon(\bar{y},\bar{z})\theta([u,v,x],z)(b)+ D([u,v,x],y)(c)),
        \end{align*} \begin{align*}
           & \varepsilon(\bar{u}+\bar{v},\bar{x})[(x,a),[(u,d),(v,e),(y,b)]_V,(z,c)]_V \\ &= \varepsilon(\bar{u}+\bar{v},\bar{x})([x,[u,v,y],z],\varepsilon(\bar{x},\bar{u}+\bar{v}+\bar{y}+\bar{z})\theta([u,v,y],z)(a) \\ &- \varepsilon(\bar{u}+\bar{v}+\bar{y},\bar{z})\theta(x,z)(\varepsilon(\bar{u},\bar{v}+\bar{y})\theta(v,y)(d)- \varepsilon(\bar{v},\bar{y})\theta(u,y)(e)+D(u,v)(b))+D(x,[u,v,y])(c)),
    \end{align*} \begin{align*}
        & \varepsilon(\bar{u}+\bar{v},\bar{x}+\bar{y})[(x,a),(y,b),[(u,d),(v,e),(z,c)]_V]_V \\ &=  \varepsilon(\bar{u}+\bar{v},\bar{x}+\bar{y})([x,y,[u,v,z]], \varepsilon(\bar{x},\bar{y}+\bar{u}+\bar{v}+\bar{z})\theta(y,[u,v,z])(a)- \varepsilon(\bar{y},\bar{u}+\bar{v}+\bar{z})\theta(x,[u,v,z])(b) \\ & + D(x,y)(\varepsilon(\bar{u},\bar{v}+\bar{z})\theta(v,z)(d)-\varepsilon(\bar{v},\bar{z})\theta(u,z)(e) + D(u,v)(c))).
    \end{align*} The first component of the sum \begin{align*}
         &[[(u,d),(v,e),(x,a)]_V,(y,b),(z,c)]_V+ \varepsilon(\bar{u}+\bar{v},\bar{x})[(x,a),[(u,d),(v,e),(y,b)]_V,(z,c)]_V \\ & + \varepsilon(\bar{u}+\bar{v},\bar{x}+\bar{y})[(x,a),(y,b),[(u,d),(v,e),(z,c)]_V]_V 
    \end{align*} is the same as $[(u,d),(v,e),[(x,a),(y,b),(z,c)]_V]_V $. For the second component,  we  group, for example,  the terms evaluated in $d$. Therefore, we have \begin{align*}
        &\varepsilon(\bar{u}+\bar{v}+\bar{x},\bar{y}+\bar{z})\varepsilon(\bar{u},\bar{v}+\bar{x})\theta(y,z)\theta(v,x)(d)-\varepsilon(\bar{u}+\bar{v},\bar{x})\varepsilon(\bar{u}+\bar{v}+\bar{y},\bar{z})\varepsilon(\bar{u},\bar{v}+\bar{y})\theta(x,z)\theta(v,y)(d) \\ &+ \varepsilon(\bar{u}+\bar{v},\bar{x}+\bar{y})\varepsilon(\bar{u},\bar{v}+\bar{z})D(x,y)\theta(v,z)(d) \\ &= \varepsilon(\bar{u},\bar{x}+\bar{y}+\bar{z}+\bar{v})(\varepsilon(\bar{v}+\bar{x},\bar{y}+\bar{z}) \theta(y,z)\theta(v,x)(d)-\varepsilon(\bar{v},\bar{x})\varepsilon(\bar{v}+\bar{y},\bar{z})\theta(x,z)\theta(v,y)(d) \\ &+\varepsilon(\bar{v},\bar{x}+\bar{y})D(x,y)\theta(v,z)(d)) \\ & = \varepsilon(\bar{u},\bar{v}+\bar{x}+\bar{y}+\bar{z})\theta(v,[x,y,z])
    \end{align*} based  on the definition of a representation of $T$. We  do the same with the other terms and then Condition \eqref{lts4} is proved. 
\end{proof}

Let $L$ be a Lie color algebra and $\varphi:L\to \mathrm{End}(V)$ be a representation. We define on $L\oplus V $ a structure of Lie color algebra with the bracket \begin{center}
    $[(x,v),(y,w)]= ([x,y],\varphi(x)(w) - \varepsilon(\bar{x},\bar{y})\varphi(y)(v)) ,$ for $(x,v),(y,w)\in \mathcal{H}(L\oplus V) $.
\end{center}
\begin{prop}
    Let $T $ be a Lie color triple system, and $(\theta,V) $ be a representation of $T$ induced by a representation $(\varphi,V)$ of $L_S(T) $. Then, $T\oplus V $ is embedded into $L_S(T)\oplus V$. 
\end{prop}

\begin{proof}
   The sum $T\oplus V $ is embedded into $L_S(T)\oplus T $ is embedded into $L_S(T) $. We will show that the following relation  $[(x,u),(y,v),(z,w)]_V=[[(x,u),(y,v)]_V,(z,w)]_V $ holds for all $(x,u),(y,v),(z,w)\in T\oplus V \subset L_S(T)\oplus V $. We have \begin{align*}
       [(x&,u),(y,v),(z,w)]_V  =([x,y,z], \varepsilon(\bar{x},\bar{y}+\bar{z})\theta(y,z)(u)-\varepsilon(\bar{y},\bar{z})\theta(x,z)(v)+ D(x,y)(w))\\ &= ([x,y,z],  \varepsilon(\bar{x},\bar{y}+\bar{z})\varepsilon(\bar{y},\bar{z})\varphi(z)\varphi(y)(u) -\varepsilon(\bar{y},\bar{z})\varepsilon(\bar{x},\bar{y})\varphi(z)\varphi(x)(v) + \varphi(x)\varphi(y)(w)\\ & \quad \quad -\varepsilon(\bar{x},\bar{y})\varphi(y)\varphi(x)(w)).   
   \end{align*} \begin{align*}
       [[(&x,u),(y,v)]_V ,(z,w)]_V =  [([x,y],\varphi(x)(v)- \varepsilon(\bar{x},\bar{y})\varphi(y)(u)),(z,w)]_V \\ =& ([x,y,z], \varphi([x,y])(w)- \varepsilon(\bar{x}+\bar{y},\bar{z})\varphi(z)\varphi(x)(v) + \varepsilon(\bar{x}+\bar{y},\bar{z})\varepsilon(\bar{x},\bar{y})\varphi(z)\varphi(y)(u)) \\ = & ([x,y,z], \varphi(x)\varphi(y)(w) - \varepsilon(\bar{x},\bar{y})\varphi(y)\varphi(x)(w) - \varepsilon(\bar{x}+\bar{y},\bar{z})\varphi(z)\varphi(x)(v) \\ &\quad \quad+ \varepsilon(\bar{x}+\bar{y},\bar{z})\varepsilon(\bar{x},\bar{y})\varphi(z)\varphi(y)(u)) \\ =& [(x,u),(y,v),(z,w)]_V.
   \end{align*} Thus, the result follows.

\end{proof}

\subsection{Representations of restricted Lie color algebras}

In \cite{F01}, a representation of a restricted Lie color algebra is defined as follows:

\begin{defi}[\cite{F01}]
    A representation of a restricted Lie color algebra $L$ is a morphism of restricted Lie color algebras $\rho:L\to \mathrm{End}(V)_G$ where $V$ is a $G$-graded vector space and $\mathrm{End}(V)$ is a Lie color algebra with the bracket $[f,g]=fg-\varepsilon(\bar{f},\bar{g})gf $ and the $p$-map is given by the $p^{th}$ power.
\end{defi}

Following the idea in \cite{E23}, we first show that under some conditions we can put on $L\oplus V $ a restricted structure. 

Let $L$ be a restricted Lie color algebra and $\varphi:L\to \mathrm{End}(V)$ be a restricted representation. We define on $L\oplus V $ a structure of Lie color algebra with the bracket \begin{center}
    $[(x,v),(y,w)]= ([x,y],\varphi(x)(w) - \varepsilon(\bar{x},\bar{y})\varphi(y)(v)) $ for all $(x,v),(y,w)\in \mathcal{H}(L\oplus V) $.
\end{center}

\begin{lem}
    Let $L$ be a Lie color algebra, and $(V,\varphi) $ a representation. Let $\gamma\in G_+, x\in L_\gamma$ then, for all  $y\in \mathcal{H}(L) $, $n\in \mathbf{N}$, we have \begin{center}
        $\varphi(ad_x^{n}(y))= \sum\limits_{k=0}^n(-1)^k\varepsilon(\bar{x},\bar{y})^k\binom{n}{k}\varphi(x)^{n-k}\varphi(y)\varphi(x)^k $.
    \end{center}
\end{lem}

\begin{proof}
   We perform by induction on $n$. The cases $n=0,1$ are immediate. Assume that the formula is true for $n\geq 1$. Let $y\in \mathcal{H}(L)$. \begin{align*}
        \varphi(ad_x^{n+1}(y)) =& \varphi(x)\varphi(ad_x^n(y))- \varepsilon(\bar{x},\bar{y})\varphi(ad_x^n(y))\varphi(x) \\ =& \varphi(x)(\sum\limits_{k=0}^n(-1)^k\varepsilon(\bar{x},\bar{y})^k\tbinom{n}{k}\varphi(x)^{n-k}\varphi(y)\varphi(x)^k ) \\  & - (\sum\limits_{k=0}^n(-1)^k\varepsilon(\bar{x},\bar{y})^{k+1}\tbinom{n}{k}\varphi(x)^{n-k}\varphi(y)\varphi(x)^k )\varphi(x) \\  =& \varphi(x)^{n+1}\varphi(y) + (-1)^{n+1}\varepsilon(\bar{x},\bar{y})^{n+1} \varphi(y)\varphi(x)^{n+1} \\ &+ \sum\limits_{k=1}^n(-1)^k\varepsilon(\bar{x},\bar{y})^k\tbinom{n}{k}\varphi(x)^{n-k+1}\varphi(y)\varphi(x)^k \\ & - \sum\limits_{k=0}^{n-1}(-1)^k\varepsilon(\bar{x},\bar{y})^{k+1}\tbinom{n}{k}\varphi(x)^{n-k}\varphi(y)\varphi(x)^{k+1} \\ = & \varphi(x)^{n+1}\varphi(y) + (-1)^{n+1}\varepsilon(\bar{x},\bar{y})^{n+1}\varphi(y)\varphi(x)^{n+1} \\ &+ \sum\limits_{k=1}^n(-1)^k\varepsilon(\bar{x},\bar{y})^k(\tbinom{n}{k}+\tbinom{n}{k-1})\varphi(x)^{n+1-k}\varphi(y)\varphi(x)^k \\ = & \sum\limits_{k=0}^{n+1}(-1)^k\varepsilon(\bar{x},\bar{y})^k\tbinom{n+1}{k}\varphi(x)^{n+1-k}\varphi(y)\varphi(x)^k .
    \end{align*} The conclusion follows.
\end{proof}

\begin{lem}\label{lem5.9}
    Let $L$ be a restricted Lie color algebra and  $(\varphi,V) $ a restricted representation of $L$. Let $\gamma\in G_+, (x,v)\in (L\oplus V)_\gamma $. Then for all $(y,w)\in \mathcal{H}(L\oplus V) $, all $n\in \mathbf{N}$, we have \begin{center}
        $\ad_{(x,v)}^n (y,w)= (\ad_x^n(y), \varphi(x)^n(w) + \sum\limits_{k=1}^n(-1)^k\varepsilon(\bar{x},\bar{y})^k\binom{n}{k}\varphi(x)^{n-k}\varphi(y)\varphi(x)^{k-1}(v)) $.
    \end{center} In particular, for $n=p$ we have \begin{center}
        $\ad_{(x,v)}^p(y,w)= (\ad_x^p(y), \varphi(x)^p(w) - \varepsilon(\bar{x},\bar{y})^p\varphi(y)\varphi^{p-1}(x)(v)). $
    \end{center}
\end{lem}

\begin{proof}
    We perform by induction on $n$. The cases $n=0,1$ are immediate. Assume that the formula is true for $n\geq 1$. Let $(y,w)\in \mathcal{H}(L\oplus V) $. \begin{align*}
        \ad_{(x,v)}^{n+1}&(y,w) =   \ad_{(x,v)}(\ad_{(x,v)}^n(y,w))\\ = & [(x,v), (\ad_x^n(y),\varphi(x)^n(w) + \sum\limits_{k=1}^n(-1)^k\varepsilon(\bar{x},\bar{y})^k\tbinom{n}{k}\varphi(x)^{n-k}\varphi(y)\varphi(x)^{k-1}(v))] \\ = & (\ad_x^{n+1}(y), \varphi(x)^{n+1}(w) + \sum\limits_{k=1}^n(-1)^k\varepsilon(\bar{x},\bar{y})^k\tbinom{n}{k}\varphi(x)^{n+1-k}\varphi(y)\varphi(x)^{k-1}(v)\\ & - \varepsilon(\bar{x},\bar{y}) \varphi(\ad_x^n(y))(v)).
    \end{align*} The second component is \begin{align*}
        & \varphi(x)^{n+1}(w) + \sum\limits_{k=1}^n(-1)^k\varepsilon(\bar{x},\bar{y})^k\tbinom{n}{k}\varphi(x)^{n+1-k}\varphi(y)\varphi(x)^{k-1}(v)\\ & \quad\quad + \sum\limits_{k=1}^{n+1}(-1)^k\varepsilon(\bar{x},\bar{y})^k\tbinom{n}{k}\varphi(x)^{n+1-k}\varphi(y)\varphi(x)^{k-1}(v) \\ & = \varphi(x)^{n+1}(w) + (-1)^{n+1}\varepsilon(\bar{x},\bar{y})^{n+1}\varphi(y)\varphi(x)^n(v) \\ &+ \sum\limits_{k=1}^n(-1)^k\varepsilon(\bar{x},\bar{y})^k\tbinom{n+1}{k}\varphi(x)^{n+1-k}\varphi(y)\varphi(x)^{k-1}(v) \\ & = \varphi(x)^{n+1}(w) + \sum\limits_{k=1}^{n+1}(-1)^k\varepsilon(\bar{x},\bar{y})^k\tbinom{n+1}{k}\varphi(x)^{n+1-k}\varphi(y)\varphi(x)^{k-1}(v).
    \end{align*} The result follows.
\end{proof}

\begin{prop}
    Let $L$ be a restricted Lie color algebra and $(V,\varphi)$ be a restricted representation of $L$. Then for any basis $(x_{i,\gamma},v_{i,\gamma})$ of $(L\oplus V)_\gamma$ there exists a $p$-map given on $(L\oplus V)_\gamma$ by the formula $(x_{i,\gamma},v_{i,\gamma})^{[p]}= (x_{i,\gamma}^{[p]}, \varphi(x_{i,\gamma})^{p-1}(v_{i,\gamma})) $,  for all $\gamma\in G_+ $.
\end{prop}

\begin{proof}
    We will use  \hyperref[thm3.5]{Theorem ~\ref*{thm3.5}}. Let $\gamma\in G_+, (x,v)\in (L\oplus V)_\gamma, (y,w)\in \mathcal{H}(L\oplus V) $. We have \begin{align*}
        \ad_{(x,v)}^p(y,w)-[(x,v)^{[p]}, (y,w)] =& (\ad_x^p(y), \varphi(x)^p(w) - \varepsilon(\bar{x},\bar{y})^p\varphi(y)\varphi(x)^{p-1}(v)) \\ & -[(x^{[p]},\varphi(x)^{p-1}(v)), (y,w)]\\  =& (\ad_{x^{[p]}}(y),  \varphi(x)^p(w) - \varepsilon(\bar{x},\bar{y})^p\varphi(y)\varphi(x)^{p-1}(v))\\ & - (\ad_{x^{[p]}}(y),  \varphi(x^{[p]})(w) - \varepsilon(\bar{x},\bar{y})^p\varphi(y)\varphi(x)^{p-1}(v)) \\ = & 0,
    \end{align*} by \hyperref[lem5.9]{Lemma ~\ref*{lem5.9}}. Therefore, the result follows from Jacobson's Theorem.
\end{proof}

\begin{rmq}
    If the center of $L\oplus V $ is trivial, then the map defined by \begin{center}
    $\begin{array}{lrcl}
(\cdot)^{[p]} : & L_\gamma\oplus V_\gamma & \longrightarrow & L_{p\gamma}\oplus V_{p\gamma} \\
    & (x,v) & \longmapsto & (x^{[p]},\varphi(x)^{p-1}(v))\end{array}$ \end{center} is a $p$-map.
\end{rmq}

\subsection{Representations of restricted Lie color triple systems}
In this section, we introduce the notion of a representation of a restricted Lie color triple system. Moreover, we provide some properties and examples.
\begin{defi}\label{d5.11}
    Let $T$ be a restricted Lie color triple system, and $(\theta,V)$ be a representation of $T$. The pair $(V,\theta)$ is called a restricted representation if, in addition, we have
    \begin{center}
    $\theta(x^{[p]},y)=\varepsilon(\bar{x},\bar{y})^{p-1}\theta(x,y) \circ \theta(x,x)^{\frac{p-1}{2}} $ 
    \end{center} for all $\gamma\in G_+, x\in T_\gamma, y\in \mathcal{H}(T) $ and
    \begin{center} 
    $\theta(x,y^{[p]})=\varepsilon(\bar{x},\bar{y})^{p-1} \theta(y,y)^{\frac{p-1}{2}}\circ\theta(x,y) $
    \end{center} for all $\gamma\in G_+, x\in \mathcal{H}(T), y\in T_\gamma  $.
\end{defi}

\begin{ex}
\begin{enumerate}
    \item  The pair $(T,\theta)$ with $\theta(x,y)(z)=\varepsilon(\bar{x}+\bar{y},\bar{z})[z,x,y] $ is a restricted representation. It is called the adjoint representation.

   \item We compute  $1$-dimensional restricted representations of the restricted Lie color triple system $L=L_{(1,1,0)}\oplus L_{(1,0,1)} \oplus L_{(0,1,1)} $, defined in Example \ref{Example1}. We have seen that the $p$-map of $L$ is the identity on the basis and that the representations of $L$ are the bilinear maps $\theta:T\times T\to \mathrm{End}(V) $ defined on the basis by \begin{center}
          $\theta(e_i,e_i)(v)= \alpha v  $   and $\theta(e_i,e_j)(v)=0 $  for $i\neq j \in \left\{1,2,3 \right\}$, with $\alpha=0,1 $. 
    \end{center}To turn it into a restricted representation it is necessary and sufficient that $\alpha= \alpha^{\frac{p+1}{2}}, $ which is the case.\\ If $(V,\theta) $ is a representation of dimension $2$ such that $V=\mathbf{K}v_1\oplus \mathbf{K}v_2 $ with $(\bar{v_1},\bar{v_2})=(1,2) $, then there exist two scalars $\alpha=0,1$ and $\beta=0,1 $ such that $\theta(e_i,e_j)=0 $ if $i\neq j$ and $\theta(e_i,e_i)(v_1)=\alpha v_1 $, $\theta(e_i,e_i)(v_2)=\beta v_2 $. In order two make this a restricted representation, it is necessary and sufficient that $\alpha,\beta $ verify $\alpha= \alpha^{\frac{p+1}{2}}, \beta= \beta^{\frac{p+1}{2}},  $ which is the case.
    \end{enumerate}
\end{ex}

We have the following result:

\begin{prop}
    Let $L$ be a restricted Lie color algebra and let $\rho: L\to \mathrm{End}(V)$ be a restricted representation of $L$. The map $\varphi$ induces a restricted representation of $L_{\mathrm{trip}} $ by setting \begin{center}
        $\theta(x,y)=\varepsilon(\bar{x},\bar{y})\varphi(y)\varphi(x) $ for all $x,y\in \mathcal{H}(T)$.\end{center}
\end{prop}

\begin{proof}
    Let $\gamma\in G_+ $, $x\in L_\gamma$ and $ y\in \mathcal{H}(L) $. We have \begin{align*}
        \theta(x^{[p]},y)=& \varepsilon(p\bar{x},\bar{y})\rho(y)\rho(x^{[p]})=\varepsilon(p\bar{x},\bar{y})\rho(y)\rho(x)^p \\ = & \varepsilon(\bar{x},\bar{y})^{p-1}\varepsilon(\bar{x},\bar{y})\rho(y)\rho(x)(\rho(x)\rho(x))^{\frac{p-1}{2}}=\varepsilon(x,y)^{\frac{p-1}{2}}\theta(x,y)\theta(x,x)^{\frac{p-1}{2}}. \end{align*} Similarly for all $x\in \mathcal{H}(L) \text{ and } y\in L_\gamma$, we have  $$\theta(x,y^{[p]})= \varepsilon(\bar{x},p\bar{y})\rho(y)^p\rho(x)=\varepsilon(\bar{x},\bar{y})^{\frac{p-1}{2}}\theta(y,y)^{\frac{p-1}{2}}\theta(x,y) .$$
\end{proof}

%\begin{prop}
  %  Let $T$ be a restricted Lie color triple system with $Z(T)=0$. Let $\theta$ be a representation $T$ on $V$  induced by a representation $\varphi $ of $L_D(T) $ on $V$. Assume the operation $[\cdot,\cdot,\cdot]_V: (T\oplus V)\times (T\oplus V)\times (T\oplus V)\to T\oplus V  $ by \begin{center}
     %   $[(x,a),(y,b),(z,c)]_V=([x,y,z], %\varepsilon(\bar{x},\bar{y}+\bar{z})\theta(y,z)(a)-\varepsilon(\bar{y},\bar{z})\theta(x,z)(b)+ D(x,y)(c)) $,
    %\end{center} then $T\oplus V$ with the $p$-map definded by \begin{center}
     %   $\begin{array}{lrcl}
%[p] : & T_+\oplus V_+ & \longrightarrow & T_+\oplus V_+ \\
 %   & (x,v) & \longmapsto & (x^{[p]},\theta(x,x)^{\frac{p-1}{2}}(v))\end{array}$
  %  \end{center} is a restricted Lie color triple system.
%\end{prop}

%\begin{proof}
 %   We can endow $L_D(T) $ with a restricted structure since the center is trivial. Then, $L_D(T)\oplus V $ is restricted, and $T\oplus V $ is stable under the $p$-map of $L_D(T)\oplus V $ (it is stable on the basis of $ T\oplus V$ and then it is stable on all $T\oplus V$ because of the formula of $((x,u)+(y,v))^{[p]} $ involving the term $s_i$ which are in $T\oplus V$) so $T\oplus V$ is restricted.
%\end{proof}

Now, let $(V,\theta) $ be a restricted representation of a restricted Lie color triple system $T$. We will prove that we can endow the semi-direct product $T\oplus V $ with a restricted structure. We provide some technical results in order to use \hyperref[thm4.16]{Theorem ~\ref*{thm4.16}}. The following propositions will be used in order to prove \hyperref[thm5.18]{Theorem ~\ref*{thm5.18}}. 

\begin{prop}\label{p5.14}
    Let $T$ be a Lie color triple system. Let  $\gamma\in G_+ ,x\in T_\gamma, y\in \mathcal{H}(T) $ and $\theta$ be a representation of $T$. Then we have \begin{align*}
      &  \theta(x,[[[\dots[y,x,x],x,x]\dots],x,x])=\\ & \sum\limits_{k=0}^n\tbinom{2n}{2k}\varepsilon(\bar{y},\bar{x})^{2(n-k)}\theta(x,x)^{n-k}\theta(x,y)\theta(x,x)^k - \sum\limits_{k=0}^{n-1}\tbinom{2n}{2k+1}\varepsilon(\bar{y},\bar{x})^{2k-1}\theta(x,x)^k\theta(y,x)\theta(x,x)^{n-k} ,
    \end{align*} and \begin{align*}
       & \theta([[\dots[y,x,x]\dots],x,x],x)= \\& \sum\limits_{k=0}^n\tbinom{2n}{2k}\varepsilon(\bar{y},\bar{x})^{2k}\theta(x,x)^k\theta(y,x)\theta(x,x)^{n-k} - \sum\limits_{k=0}^{n-1}\tbinom{2n}{2k+1}\varepsilon(\bar{y},\bar{x})^{2(n-k)+1}\theta(x,x)^{n-k}\theta(x,y)\theta(x,x)^k ,
    \end{align*} where $x$ occurs $2n$ times in the brackets.
\end{prop}

\begin{proof}
   We perform by induction on $n$. The initialization is clear. Assume that the formulas are true for $n\geq 1$.

 By definition of a representation, we have \begin{align*}
     &\theta(x,[[\dots[y,x,x],x,x]\dots],x,x])= \varepsilon(\bar{y},\bar{x})^2\theta(x,x)\theta(x,[\dots[[y,x,x]\dots],x,x]) 
     \\& 
-2\varepsilon(\bar{x},\bar{y})\theta([\dots[[y,x,x]\dots],x,x],x)\theta(x,x)
+\theta(x,[\dots[[y,x,x]\dots],x,x])\theta(x,x) .
 \end{align*} Thus, we have 
    \begin{align*}
     &\theta(x,[[\dots[y,x,x],x,x]\dots],x,x]) = \sum\limits_{k=0}^n\tbinom{2n}{2k}\varepsilon(\bar{y},\bar{x})^{2(n-k)+2}\theta(x,x)^{n+1-k}\theta(x,y)\theta(x,x)^k \\ & - \sum\limits_{k=0}^{n-1}\tbinom{2n}{2k+1}\varepsilon(\bar{y},\bar{x})^{2k+1}\theta(x,x)^{k+1}\theta(y,x)\theta(x,x)^{n-k} 
     \\ &
     - 2\sum\limits_{k=0}^n\tbinom{2n}{2k}\varepsilon(\bar{y},\bar{x})^{2k-1}\theta(x,x)^k\theta(y,x)\theta(x,x)^{n+1-k} \\ & +2 \sum\limits_{k=0}^{n-1}\tbinom{2n}{2k+1}\varepsilon(\bar{y},\bar{x})^{2(n-k)}\theta(x,x)^{n-k}\theta(x,y)\theta(x,x)^{k+1} 
     \\ &
     + \sum\limits_{k=0}^n\tbinom{2n}{2k}\varepsilon(\bar{y},\bar{x})^{2(n-k)}\theta(x,x)^{n-k}\theta(x,y)\theta(x,x)^{k+1} \\ &- \sum\limits_{k=0}^{n-1}\tbinom{2n}{2k+1}\varepsilon(\bar{y},\bar{x})^{2k-1}\theta(x,x)^k\theta(y,x)\theta(x,x)^{n+1-k} \\ & = \sum\limits_{k=0}^n\tbinom{2n}{2k}\varepsilon(\bar{y},\bar{x})^{2(n-k)+2}\theta(x,x)^{n+1-k}\theta(x,y)\theta(x,x)^k  
     \\ &
     + 2 \sum\limits_{k=1}^{n}\tbinom{2n}{2k-1}\varepsilon(\bar{y},\bar{x})^{2(n+1-k)}\theta(x,x)^{n+1-k}\theta(x,y)\theta(x,x)^{k} \\ & +\sum\limits_{k=1}^{n+1}\tbinom{2n}{2k-2}\varepsilon(\bar{y},\bar{x})^{2(n+1-k)}\theta(x,x)^{n+1-k}\theta(x,y)\theta(x,x)^{k} \\ & - \sum\limits_{k=1}^{n}\tbinom{2n}{2k-1}\varepsilon(\bar{y},\bar{x})^{2k-1}\theta(x,x)^{k}\theta(y,x)\theta(x,x)^{n+1-k} \\ & -2\sum\limits_{k=0}^n\tbinom{2n}{2k}\varepsilon(\bar{y},\bar{x})^{2k-1}\theta(x,x)^k\theta(y,x)\theta(x,x)^{n+1-k}\\ & -  \sum\limits_{k=0}^{n-1}\tbinom{2n}{2k+1}\varepsilon(\bar{y},\bar{x})^{2k-1}\theta(x,x)^k\theta(y,x)\theta(x,x)^{n+1-k} \\  & =  \sum\limits_{k=1}^{n}( \tbinom{2n}{2k}+2\tbinom{2n}{2k-1}+ \tbinom{2n}{2k-2})\varepsilon(\bar{y},\bar{x})^{2(n+1-k)}\theta(x,x)^{n+1-k}\theta(x,y)\theta(x,x)^{k} \\ &+\varepsilon(\bar{y},\bar{x})^{2(n+1)}\theta(x,x)^{n+1}\theta(x,y)+\theta(x,y)\theta(x,x)^{n+1} \\ &  - \sum\limits_{k=1}^{n-1}(\tbinom{2n}{2k-1} +2 \tbinom{2n}{2k} + \tbinom{2n}{2k+1})\varepsilon(\bar{y},\bar{x})^{2k-1}\theta(x,x)^{k}\theta(y,x)\theta(x,x)^{n+1-k} \\ &- (\tbinom{2n}{1}+2)\varepsilon(\bar{x},\bar{y})\theta(y,x)\theta(x,x)^{n+1} - (2+\tbinom{2n}{2n-1})\varepsilon(\bar{y},\bar{x})^{2n-1}\theta(x,x)^n\theta(y,x)\theta(x,x) \\  &= \sum\limits_{k=0}^{n+1}\tbinom{2(n+1)}{2k}\varepsilon(\bar{y},\bar{x})^{2(n+1-k)}\theta(x,x)^{n+1-k}\theta(x,y)\theta(x,x)^k \\ & - \sum\limits_{k=0}^{n}\tbinom{2(n+1)}{2k+1}\varepsilon(\bar{y},\bar{x})^{2k-1}\theta(x,x)^k\theta(y,x)\theta(x,x)^{n+1-k} . 
    \end{align*} 
    Using  Conditions \eqref{rep2} and \eqref{rep3} of  \hyperref[d5.2]{Definition ~\ref*{d5.2}}, we have \begin{align*}
        \theta(&[[\dots[y,x,x],x,x]\dots],x,x],x) = -\varepsilon(\bar{y},\bar{x})^3\theta(x,x)\theta(x,[[\dots[y,x,x]\dots],x,x]) \\ &+ \varepsilon(\bar{y},\bar{x})^2\theta(x,x)\theta([[\dots[y,x,x]\dots],x,x],x)+ \varepsilon(\bar{y},\bar{x})\theta(x,[[\dots[y,x,x]\dots],x,x])\theta(x,x) \\ & - \theta([[\dots[y,x,x]\dots],x,x],x)\theta(x,x) - \varepsilon(\bar{y},\bar{x})\theta(x,[[\dots[y,x,x],x,x]\dots],x,x]).
    \end{align*} So \begin{align*}
        \theta(&[[\dots[y,x,x],x,x]\dots],x,x],x) = \\ & -\sum\limits_{k=0}^n\tbinom{2n}{2k}\varepsilon(\bar{y},\bar{x})^{2(n-k)+3}\theta(x,x)^{n+1-k}\theta(x,y)\theta(x,x)^k  \\ &+ \sum\limits_{k=0}^{n-1}\tbinom{2n}{2k+1}\varepsilon(\bar{y},\bar{x})^{2k+2}\theta(x,x)^{k+1}\theta(y,x)\theta(x,x)^{n-k} \\ &+\sum\limits_{k=0}^n\tbinom{2n}{2k}\varepsilon(\bar{y},\bar{x})^{2k+2}\theta(x,x)^{k+1}\theta(y,x)\theta(x,x)^{n-k} \\ &  - \sum\limits_{k=0}^{n-1}\tbinom{2n}{2k+1}\varepsilon(\bar{y},\bar{x})^{2(n+1-k)+1}\theta(x,x)^{n+1-k}\theta(x,y)\theta(x,x)^k \\ & + \sum\limits_{k=0}^n\tbinom{2n}{2k}\varepsilon(\bar{y},\bar{x})^{2(n-k)+1}\theta(x,x)^{n-k}\theta(x,y)\theta(x,x)^{k+1}\\ & - \sum\limits_{k=0}^{n-1}\tbinom{2n}{2k+1}\varepsilon(\bar{y},\bar{x})^{2k}\theta(x,x)^k\theta(y,x)\theta(x,x)^{n+1-k} \\ & - \sum\limits_{k=0}^n\tbinom{2n}{2k}\varepsilon(\bar{y},\bar{x})^{2k}\theta(x,x)^k\theta(y,x)\theta(x,x)^{n+1-k}  \\ &+ \sum\limits_{k=0}^{n-1}\tbinom{2n}{2k+1}\varepsilon(\bar{y},\bar{x})^{2(n-k)+1}\theta(x,x)^{n-k}\theta(x,y)\theta(x,x)^{k+1} \\ & - \sum\limits_{k=0}^{n+1}\tbinom{2(n+1)}{2k}\varepsilon(\bar{y},\bar{x})^{2(n+1-k)+1}\theta(x,x)^{n+1-k}\theta(x,y)\theta(x,x)^k \\ &+ \sum\limits_{k=0}^{n}\tbinom{2(n+1)}{2k+1}\varepsilon(\bar{y},\bar{x})^{2k}\theta(x,x)^k\theta(y,x)\theta(x,x)^{n+1-k} \\ &= \sum\limits_{k=1}^{n-1}C\varepsilon(\bar{y},\bar{x})^{2(n+1-k)}\theta(x,x)^{n+1-k}\theta(x,y)\theta(x,x)^k \\ & -(\tbinom{2n}{1}+2)\varepsilon(\bar{y},\bar{x})^{2(n+1)+1}\theta(x,x)^{n+1}\theta(x,y) \\ &+(-1+\tbinom{2n}{2n-2} -\tbinom{2n+2}{2n}+\tbinom{2n}{2n-1}) \varepsilon(\bar{y},\bar{x})^3\theta(x,x)\theta(x,y)\theta(x,x)^n \\  &+ \sum\limits_{k=1}^{n-1}C'\varepsilon(\bar{y},\bar{x})^{2k}\theta(x,x)^k\theta(y,x)\theta(x,x)^{n+1-k} \\ & + (\tbinom{2n}{2n-1}+ \tbinom{2n}{2n-2} + \tbinom{2n+2}{2n+1}-1)\varepsilon(\bar{y},\bar{x})^{2n}\theta(x,x)^n\theta(y,x)\theta(x,x) \\ &+ \varepsilon(\bar{y},\bar{x})^{2(n+1)}\theta(x,x)^{n+1}\theta(y,x) - \theta(y,x)\theta(x,x)^{n+1} \\ & = \sum\limits_{k=0}^{n+1}\tbinom{2(n+1)}{2k}\varepsilon(\bar{y},\bar{x})^{2k}\theta(x,x)^k\theta(y,x)\theta(x,x)^{n+1-k} \\ &  - \sum\limits_{k=0}^{n}\tbinom{2(n+1)}{2k+1}\varepsilon(\bar{y},\bar{x})^{2(n+1-k)+1}\theta(x,x)^{n+1-k}\theta(x,y)\theta(x,x)^k, 
    \end{align*} where $$C=(\tbinom{2n}{2k-1}+\tbinom{2n}{2k-2}-\tbinom{2n}{2k}-\tbinom{2n}{2k+1}-\tbinom{2(n+1)}{2k})= \tbinom{2(n+1)}{2k} $$ and $$C'= (\tbinom{2n}{2k-1} +\tbinom{2n}{2k-2}-\tbinom{2n}{2k+1} -\tbinom{2n}{2k}+ \tbinom{2(n+1)}{2k+1})= \tbinom{2(n+1)}{2k+1}. $$ Then the result follows.
\end{proof}

\begin{prop}\label{p5.15}
    Let $T$ be a Lie color triple system, $(y,u)\in \mathcal{H}(T\oplus V)$ and $(x,v)\in \mathcal{H}((T\oplus V)_+) $. Then we have \begin{align*}
        &((y,u),(x,v),\dots,(x,v))=\\ & ((y,x,\dots,x), \varepsilon(\bar{y},\bar{x})^{2n}\theta(x,x)^n(u)- \sum\limits_{k=0}^{n-1}\tbinom{2n}{2k+1}\varepsilon(\bar{y},\bar{x})^{2k}\theta(x,x)^k\theta(y,x)\theta(x,x)^{n-1-k}(v) \\&+ \sum\limits_{k=0}^{n-1}\tbinom{2n}{2(k+1)}\varepsilon(\bar{y},\bar{x})^{2(n-k)-1}\theta(x,x)^{n-k-1}\theta(x,y)\theta(x,x)^k(v) )  ,
    \end{align*} where $x$ occurs $2n$ times in the bracket.
\end{prop}

\begin{proof}
   We perform by induction on $n$. The initialization is clear. Assume that the formula holds for $n\geq 1$. We have, by definition of the bracket in $T\oplus V$, that \begin{align*}
        [\dots[[[(y,u),(x,v),(x,v)]&,(x,v),(x,v)]\dots](x,v),(x,v)]=((y,x,\dots,x), \varepsilon(\bar{y},\bar{x})^2\theta(x,x)(a) \\ &-2\theta([[\dots[y,x,x]\dots]x,x],x)(v)+\varepsilon(\bar{y},\bar{x})\theta(x,[\dots[[y,x,x]\dots],x,x])(v)),
    \end{align*} where in left hand side $(x,v)$ occurs $2(n+1)$ times and $a$ is the second component of  $((y,u),(x,v),\dots,(x,v))$, where $(x,v)$ occurs $2n$ times.\\ So, the second component of $ [\dots[[[(y,u),(x,v),(x,v)],(x,v),(x,v)]\dots](x,v),(x,v)] $ is \begin{align*}
        \varepsilon(\bar{y},\bar{x})^2&\theta(x,x)(\varepsilon(\bar{y},\bar{x})^{2n}\theta(x,x)^n(u)- \sum\limits_{k=0}^{n-1}\tbinom{2n}{2k+1}\varepsilon(\bar{y},\bar{x})^{2k}\theta(x,x)^k\theta(y,x)\theta(x,x)^{n-1-k}(v)  \\ & +\sum\limits_{k=0}^{n-1}\tbinom{2n}{2(k+1)}\varepsilon(\bar{y},\bar{x})^{2(n-k)-1}\theta(x,x)^{n-k-1}\theta(x,y)\theta(x,x)^k(v)) \\& -2\theta([[\dots[y,x,x]\dots]x,x],x)(v) +\varepsilon(\bar{y},\bar{x})\theta(x,[\dots[[y,x,x]\dots],x,x])(v)\\ & = \varepsilon(\bar{y},\bar{x})^{2(n+1)}\theta(x,x)^{n+1}(u)  - \sum\limits_{k=0}^{n-1}\tbinom{2n}{2k+1}\varepsilon(\bar{y},\bar{x})^{2k+2}\theta(x,x)^{k+1}\theta(y,x)\theta(x,x)^{n-1-k}(v) \\& +  \sum\limits_{k=0}^{n-1}\tbinom{2n}{2k+2}\varepsilon(\bar{y},\bar{x})^{2(n-k)+1}\theta(x,x)^{n-k}\theta(x,y)\theta(x,x)^k(v) \\ & -2 \sum\limits_{k=0}^n\tbinom{2n}{2k}\varepsilon(\bar{y},\bar{x})^{2k}\theta(x,x)^k\theta(y,x)\theta(x,x)^{n-k}(v) \\ & + 2\sum\limits_{k=0}^{n-1}\tbinom{2n}{2k+1}\varepsilon(\bar{y},\bar{x})^{2(n-k)+1}\theta(x,x)^{n-k}\theta(x,y)\theta(x,x)^k \\ & + \sum\limits_{k=0}^n\tbinom{2n}{2k}\varepsilon(\bar{y},\bar{x})^{2(n-k)+1}\theta(x,x)^{n-k}\theta(x,y)\theta(x,x)^k \\ & - \sum\limits_{k=0}^{n-1}\tbinom{2n}{2k+1}\varepsilon(\bar{y},\bar{x})^{2k}\theta(x,x)^k\theta(y,x)\theta(x,x)^{n-k} \\ & = \varepsilon(\bar{y},\bar{x})^{2(n+1)}\theta(x,x)^{n+1}(u) \\ & -\sum\limits_{k=0}^{n-1}(\tbinom{2n}{2k-1}+ 2\tbinom{2n}{2k}+ \tbinom{2n}{2k+1})\varepsilon(\bar{y},\bar{x})^{2k}\theta(x,x)^k\theta(y,x)\theta(x,x)^{n-k}(v) \\ & - (\tbinom{2n}{2n-1}+2)\varepsilon(\bar{y},\bar{x})^{2n}\theta(x,x)^n\theta(y,x)(v) - (\tbinom{2n}{1}+2)\theta(y,x)\theta(x,x)^n(v) \\ & +\sum\limits_{k=0}^{n-1}(\tbinom{2n}{2k}+ 2\tbinom{2n}{2k+1}+ \tbinom{2n}{2k+2})\varepsilon(\bar{y},\bar{x})^{2(n-k)+1}\theta(x,x)^{n-k}\theta(x,y)\theta(x,x)^k(v)\\ & + \theta(x,y)\theta(x,x)^n(v) \\ & = \varepsilon(\bar{y},\bar{x})^{2(n+1)}\theta(x,x)^{n+1}(u)- \sum\limits_{k=0}^{n}\tbinom{2(n+1)}{2k+1}\varepsilon(\bar{y},\bar{x})^{2k}\theta(x,x)^k\theta(y,x)\theta(x,x)^{n-k}(v)  \\ &+ \sum\limits_{k=0}^{n}\tbinom{2(n+1)}{2(k+1)}\varepsilon(\bar{y},\bar{x})^{2(n+1-k)-1}\theta(x,x)^{n-k}\theta(x,y)\theta(x,x)^k(v).
   \end{align*} The result follows.
\end{proof}

\begin{prop}\label{p5.16}
    Let $T$ be a Lie color triple system, $y,z\in \mathcal{H}(T) $ and $x\in \mathcal{H}(T_+)$. Then we have \begin{align*}
       & \theta([[\dots[y,x,x]\dots],x,x],z)=\\ & \varepsilon(\bar{x},\bar{z})^{2n-1}\varepsilon(\bar{y},\bar{z})\theta(x,z)\sum\limits_{k=1}^n \tbinom{2n}{2k}\varepsilon(\bar{y},\bar{x})^{2k-1}\theta(x,x)^{k-1}\theta(y,x)\theta(x,x)^{n-k} \\ &- \varepsilon(\bar{x},\bar{z})^{2n-1}\varepsilon(\bar{y},\bar{z})\theta(x,z)\sum\limits_{k=0}^{n-1}\tbinom{2n}{2k+1}\varepsilon(\bar{y},\bar{x})^{2(n-k)}\theta(x,x)^{n-k-1}\theta(x,y)\theta(x,x)^k \\ &+ \varepsilon(\bar{x},\bar{z})^{2n}\theta(y,z)\theta(x,x)^n ,
    \end{align*} where $x$ occurs $2n$ times in the bracket.
\end{prop}

\begin{proof}
   We perform by induction on $n$. The initialization is clear. Assume that the formula holds for $n\geq 1$. We have 
    \begin{align*}
        \theta(&[[\dots[[y,x,x],x,x]\dots],x,x],z)= \varepsilon(\bar{y},\bar{z}+\bar{x})\varepsilon(\bar{x},\bar{z})^{2n+1}\theta(x,z)\theta([[\dots[y,x,x]\dots],x,x],x) \\ & -  2\varepsilon(\bar{y},\bar{z}+\bar{x})\varepsilon(\bar{x},\bar{z})^{2n+1}\varepsilon(\bar{y},\bar{x})\theta(x,z)\theta(x,[[\dots[y,x,x]\dots],x,x]) \\ &+\varepsilon(\bar{x},\bar{z})^2\theta([[\dots[y,x,x]\dots],x,x],z)\theta(x,x) \\ & = \varepsilon(\bar{y},\bar{z}+\bar{x})\varepsilon(\bar{x},\bar{z})^{2n+1}\theta(x,z)(\sum\limits_{k=0}^n\tbinom{2n}{2k}\varepsilon(\bar{y},\bar{x})^{2k}\theta(x,x)^k\theta(y,x)\theta(x,x)^{n-k}) \\ & - \varepsilon(\bar{y},\bar{z}+\bar{x})\varepsilon(\bar{x},\bar{z})^{2n+1}\theta(x,z)(\sum\limits_{k=0}^{n-1}\tbinom{2n}{2k+1}\varepsilon(\bar{y},\bar{x})^{2(n-k)+1}\theta(x,x)^{n-k}\theta(x,y)\theta(x,x)^k) \\ & - 2\varepsilon(\bar{y},\bar{z}+\bar{x})\varepsilon(\bar{x},\bar{z})^{2n+1}\varepsilon(\bar{y},\bar{x})\theta(x,z)(\sum\limits_{k=0}^n\tbinom{2n}{2k}\varepsilon(\bar{y},\bar{x})^{2(n-k)}\theta(x,x)^{n-k}\theta(x,y)\theta(x,x)^k)    \\ & +  2\varepsilon(\bar{y},\bar{z}+\bar{x})\varepsilon(\bar{x},\bar{z})^{2n+1}\varepsilon(\bar{y},\bar{x})\theta(x,z)(\sum\limits_{k=0}^{n-1}\tbinom{2n}{2k+1}\varepsilon(\bar{y},\bar{x})^{2(n-k)+1}\theta(x,x)^{n-k}\theta(x,y)\theta(x,x)^k ) \\ & + \varepsilon(\bar{x},\bar{z})^2(\varepsilon(\bar{x},\bar{z})^{2n-1}\varepsilon(\bar{y},\bar{z})\theta(x,z)\sum\limits_{k=1}^n \tbinom{2n}{2k}\varepsilon(\bar{y},\bar{x})^{2k-1}\theta(x,x)^{k-1}\theta(y,x)\theta(x,x)^{n-k})\theta(x,x) \\ &  - \varepsilon(\bar{x},\bar{z})^2(\varepsilon(\bar{x},\bar{z})^{2n-1}\varepsilon(\bar{y},\bar{z})\theta(x,z)\sum\limits_{k=0}^{n-1}\tbinom{2n}{2k+1}\varepsilon(\bar{y},\bar{x})^{2(n-k)}\theta(x,x)^{n-k-1}\theta(x,y)\theta(x,x)^k)\theta(x,x) \\ & + \varepsilon(\bar{x},\bar{z})^2(\varepsilon(\bar{x},\bar{z})^{2n}\theta(y,z)\theta(x,x)^n)\theta(x,x) \\ & =\varepsilon(\bar{x},\bar{z})^{2n+1}\varepsilon(\bar{y},\bar{z})\theta(x,z)\sum\limits_{k=1}^{n+1} \tbinom{2n+2}{2k}\varepsilon(\bar{y},\bar{x})^{2k-1}\theta(x,x)^{k-1}\theta(y,x)\theta(x,x)^{n+1-k} \\ &  - \varepsilon(\bar{x},\bar{z})^{2n+1}\varepsilon(\bar{y},\bar{z})\theta(x,z)\sum\limits_{k=0}^{n}\tbinom{2n+2}{2k+1}\varepsilon(\bar{y},\bar{x})^{2(n+1-k)}\theta(x,x)^{n-k}\theta(x,y)\theta(x,x)^k  \\ & + \varepsilon(\bar{x},\bar{z})^{2n+2}\theta(y,z)\theta(x,x)^{n+1}.
        \end{align*} Then the  result follows.
\end{proof}

\begin{prop}\label{p5.17}
    Let $T$ be a Lie color triple system, $(y,u),(z,w)\in \mathcal{H}(T\oplus V)$ and $(x,v)\in \mathcal{H}((T\oplus V)_+)$. Then we have 
    \begin{align*}
       & ((y,u),(x,v),\dots,(x,v),(z,w))= 
       \\ &((y,x,\dots,x,z), \varepsilon(\bar{y},\bar{x})^{2n+1}\varepsilon(\bar{y},\bar{z})\varepsilon(\bar{x},\bar{z})^{2n}\theta(x,z)\theta(x,x)^n(u)\\ & - \varepsilon(\bar{y},\bar{z})\varepsilon(\bar{x},\bar{z})^{2n}\theta(x,z)\sum\limits_{k=1}^n\tbinom{2n+1}{2k}\varepsilon(\bar{y},\bar{x})^{2k-1}\theta(x,x)^{k-1}\theta(y,x)\theta(x,x)^{n-k}(v)\\ & - \varepsilon(\bar{x},\bar{z})^{2n+1}\theta(y,z)\theta(x,x)^n(v) \\ &+ \varepsilon(\bar{x},\bar{z})^{2n}\varepsilon(\bar{y},\bar{z})\theta(x,z)\sum\limits_{k=0}^{n-1}\tbinom{2n+1}{2k+2}\varepsilon(\bar{y},\bar{x})^{2(n-k)}\theta(x,x)^{n-k-1}\theta(x,y)\theta(x,x)^k(v)\\ & +\sum\limits_{k=0}^n \tbinom{2n+1}{2k+1}\varepsilon(\bar{y},\bar{x})^{2(n-k)+1}\theta(x,x)^{n-k}\theta(x,y)\theta(x,x)^k(w)\\ &-\sum\limits_{k=0}^n \tbinom{2n+1}{2k+1}\varepsilon(\bar{y},\bar{x})^{2k}\theta(x,x)^{k}\theta(x,y)\theta(x,x)^{n-k}(w) ) ,
    \end{align*} where $(x,v)$ appears $2n+1$ times. In particular, \begin{align*}
        &((y,u),(x,v),\dots,(x,v),(z,w)) = ((y,x,\dots,x,z), \varepsilon(\bar{y},\bar{x})^p\varepsilon(\bar{y},\bar{z})\varepsilon(\bar{x},\bar{z})^{p-1}\theta(x,z)\theta(x,x)^{\frac{p-1}{2}}(u)\\ & - \varepsilon(\bar{x},\bar{z})^p\theta(y,z)\theta(x,x)^{\frac{p-1}{2}}(v) + \varepsilon(\bar{y},\bar{x})\theta(x,y)\theta(x,x)^{\frac{p-1}{2}}(w) - \varepsilon(\bar{y},\bar{x})^{p-1}\theta(x,x)^{\frac{p-1}{2}}\theta(y,x)(w))  ,
    \end{align*} where $(x,v)$ occurs $p$ times.
\end{prop}

\begin{proof}
    We just apply Propositions \hyperref[p5.14]{\ref{p5.14}}, \ref{p5.15} and \ref{p5.16}. We have \begin{align*}
    ((y,u),(x,v&),\dots,(x,v),(z,w))= [((y,u),(x,v),\dots,(x,v)),(x,v),(z,w)] \\ & = ((y,x,\dots,x,z), \varepsilon(\bar{y}+2n\bar{x},\bar{x}+\bar{z})\theta(x,z)(a) - \varepsilon(\bar{x},\bar{z})\theta([[\dots[y,x,x]\dots],x,x],z)(v) \\ &+ D([[\dots[y,x,x]\dots],x,x],x)(w),  
    \end{align*} where $a$ is the second component of $((y,u),(x,v),\dots,(x,v)) $. So the second component of $ ((y,u),(x,v),\dots,(x,v),(z,w)) $ is \begin{align*}
         \varepsilon&(\bar{y}+2n\bar{x},\bar{x}+\bar{z})\theta(x,z)(\varepsilon(\bar{y},\bar{x})^{2n}\theta(x,x)^n(u)- \sum\limits_{k=0}^{n-1}\tbinom{2n}{2k+1}\varepsilon(\bar{y},\bar{x})^{2k}\theta(x,x)^k\theta(y,x)\theta(x,x)^{n-1-k}(v)) \\ & + \varepsilon(\bar{y}+2n\bar{x},\bar{x}+\bar{z})\theta(x,z)(\sum\limits_{k=0}^{n-1}\tbinom{2n}{2(k+1)}\varepsilon(\bar{y},\bar{x})^{2(n-k)-1}\theta(x,x)^{n-k-1}\theta(x,y)\theta(x,x)^k(v) ) \\ & - \varepsilon(\bar{x},\bar{z})\varepsilon(\bar{x},\bar{z})^{2n-1}\varepsilon(\bar{y},\bar{z})\theta(x,z)\sum\limits_{k=1}^n \tbinom{2n}{2k}\varepsilon(\bar{y},\bar{x})^{2k-1}\theta(x,x)^{k-1}\theta(y,x)\theta(x,x)^{n-k}(v) \\ & + \varepsilon(\bar{x},\bar{z})^{2n}\varepsilon(\bar{y},\bar{z})\theta(x,z)\sum\limits_{k=0}^{n-1}\tbinom{2n}{2k+1}\varepsilon(\bar{y},\bar{x})^{2(n-k)}\theta(x,x)^{n-k-1}\theta(x,y)\theta(x,x)^k(v) \\ & - \varepsilon(\bar{x},\bar{z})^{2n+1}\theta(y,z)\theta(x,x)^n(v) + \varepsilon(\bar{y},\bar{x})\sum\limits_{k=0}^n\tbinom{2n}{2k}\varepsilon(\bar{y},\bar{x})^{2(n-k)}\theta(x,x)^{n-k}\theta(x,y)\theta(x,x)^k(w) \\ & - \varepsilon(\bar{y},\bar{x})\sum\limits_{k=0}^{n-1}\tbinom{2n}{2k+1}\varepsilon(\bar{y},\bar{x})^{2k-1}\theta(x,x)^k\theta(y,x)\theta(x,x)^{n-k}(w) \\ & -\sum\limits_{k=0}^n\tbinom{2n}{2k}\varepsilon(\bar{y},\bar{x})^{2k}\theta(x,x)^k\theta(y,x)\theta(x,x)^{n-k}(w)\\ & + \sum\limits_{k=0}^{n-1}\tbinom{2n}{2k+1}\varepsilon(\bar{y},\bar{x})^{2(n-k)+1}\theta(x,x)^{n-k}\theta(x,y)\theta(x,x)^k(w) \\ & =\varepsilon(\bar{y},\bar{x})^{2n+1}\varepsilon(\bar{y},\bar{z})\varepsilon(\bar{x},\bar{z})^{2n}\theta(x,z)\theta(x,x)^n(u) \\ &  - \varepsilon(\bar{y},\bar{z})\varepsilon(\bar{x},\bar{z})^{2n}\theta(x,z)\sum\limits_{k=1}^n\tbinom{2n+1}{2k}\varepsilon(\bar{y},\bar{x})^{2k-1}\theta(x,x)^{k-1}\theta(y,x)\theta(x,x)^{n-k}(v) \\ &  - \varepsilon(\bar{x},\bar{z})^{2n+1}\theta(y,z)\theta(x,x)^n(v)\\ &  + \varepsilon(\bar{x},\bar{z})^{2n}\varepsilon(\bar{y},\bar{z})\theta(x,z)\sum\limits_{k=0}^{n-1}\tbinom{2n+1}{2k+2}\varepsilon(\bar{y},\bar{x})^{2(n-k)}\theta(x,x)^{n-k-1}\theta(x,y)\theta(x,x)^k(v) \\ & +\sum\limits_{k=0}^n \tbinom{2n+1}{2k+1}\varepsilon(\bar{y},\bar{x})^{2(n-k)+1}\theta(x,x)^{n-k}\theta(x,y)\theta(x,x)^k(w) \\ & -\sum\limits_{k=0}^n \tbinom{2n+1}{2k+1}\varepsilon(\bar{y},\bar{x})^{2k}\theta(x,x)^{k}\theta(x,y)\theta(x,x)^{n-k}(w).
    \end{align*} In particular, for $n=\frac{p-1}{2}$, some terms vanish because of the characteristic $p$ and the result follows.
\end{proof}

\begin{thm}\label{thm5.18}
     Let $T$ be a restricted Lie color triple system over a field $\mathbf{K} $ of characteristic $p>3 $. Let $(V,\theta)$ be a restricted representation of $T$. Consider the operation $[\cdot,\cdot,\cdot]_V: (T\oplus V)\times (T\oplus V)\times (T\oplus V)\to T\oplus V  $ defined on homogeneous elements by 
        \begin{equation*}[(x,a),(y,b),(z,c)]_V=([x,y,z], \varepsilon(\bar{x},\bar{y}+\bar{z})\theta(y,z)(a)-\varepsilon(\bar{y},\bar{z})\theta(x,z)(b)+ D(x,y)(c)). \end{equation*}
     Then there exists a $p$-map on $T\oplus V$ given by the formula \begin{center}
        $(x_{i,\gamma},v_{i,\gamma})^{[p]}=(x_{i,\gamma}^{[p]},\theta(x_{i,\gamma},x_{i,\gamma})^{\frac{p-1}{2}}(v_{i,\gamma}))$
    \end{center} for any basis $(x_{i,\gamma},v_{i,\gamma}) $ of $T_\gamma $ and for all $\gamma\in G_+ $.
\end{thm}

\begin{proof}
    Using \hyperref[p5.17]{Proposition ~\ref*{p5.17}}, we see that one can use Jacobson's Theorem. Indeed, we have \begin{align*}
        [(y&,u),(x^{[p]},\theta(x,x)^{\frac{p-1}{2}}(v)),(z,w)] \\ &= ([y,x^{[p]},z], \varepsilon(\bar{y},p\bar{x}+\bar{z})\theta(x^{[p]},z)(u) - \varepsilon(p\bar{x},\bar{z})\theta(y,z)\theta(x,x)^{\frac{p-1}{2}}(v)\\& \quad\quad\quad\quad\quad\quad\quad\quad\quad + \varepsilon(\bar{y},p\bar{x})\theta(x^{[p]},y)(w) - \theta(y,x^{[p]})(w)) \\ & = ([y,x^{[p]},z], \varepsilon(\bar{y},\bar{x})^p\varepsilon(\bar{y},\bar{z})\varepsilon(\bar{x},\bar{z})^{p-1}\theta(x,z)\theta(x,x)^{\frac{p-1}{2}}(u) - \varepsilon(\bar{x},\bar{z})^p\theta(y,z)\theta(x,x)^{\frac{p-1}{2}}(v) \\ &\quad\quad\quad\quad\quad\quad\quad\quad\quad + \varepsilon(\bar{y},\bar{x})\theta(x,y)\theta(x,x)^{\frac{p-1}{2}}(w) - \varepsilon(\bar{y},\bar{x})^{p-1}\theta(x,x)^{\frac{p-1}{2}}\theta(y,x)(w)) \\ & = ((y,u),(x,v),\dots,(x,v),(z,w)).
    \end{align*} So, for any basis $((x_i,v_i))_{i\in I} $ of $T_\gamma $ for $\gamma\in G_+ $, there exists a $p$-map $(\cdot)^{[p]}$ on $T\oplus V $ such that $(x_i,v_i)^{[p]}= (x_i^{[p]},\theta(x_i,x_i)^{\frac{p-1}{2}}(v_i)) $.
\end{proof}

\begin{rmq}
    If the center of $T\oplus V $ is trivial, then, for all $\gamma\in G_+ $, the map \begin{center}
    $\begin{array}{lrcl}
(\cdot)^{[p]} : & T_\gamma\oplus V_\gamma & \longrightarrow & T_\gamma\oplus V_\gamma \\
    & (x,v) & \longmapsto & (x^{[p]},\theta(x,x)^{\frac{p-1}{2}}(v))\end{array}$ \end{center} is a $p$-map.
\end{rmq}

\begin{ex}
    Let  $L=L_{(1,1,0)}\oplus L_{(1,0,1)} \oplus L_{(0,1,1)} $ be the restricted Lie color triple system defined above, and $(V,\theta) $ be a restricted representation of dimension $1$ of $L$ with $V=\mathbf{K}v $, $v$ of degree $1$. Then $((e_1,0),(e_2,0),(e_3,0),(0,v)) $ is a basis of $L\oplus V $. Consider the map $(\cdot)^{[p]} $ defined on the basis by $(e_i,0)^{[p]}= (e_i,0) $, $(0,v)^{[p]}= (0,0) $, $(e_1,v)=(e_1,\theta(e_1,e_1)^{\frac{p-1}{2}}(v))=(e_1,\alpha^{\frac{p-1}{2}}v)=(e_1,v) $ because $\alpha^{\frac{p+1}{2}}=\alpha $. The center of $L$ is trivial so as the center of $L\oplus V $. Indeed, if $(x,a) $ is in the center of $L\oplus V$, then $x$ is in the center of $L$ and then $x=0$, and $\theta(y,z)(a)=0 $ for all $y,z\in L $. With $y=z=e_1 $, we see that $a=0$. So, for all $\gamma\in G_+ $ the map \begin{center}
    $\begin{array}{lrcl}
(\cdot)^{[p]} : & L_\gamma\oplus V_\gamma & \longrightarrow & L_\gamma\oplus V_\gamma \\
    & (x,v) & \longmapsto & (x^{[p]},\theta(x,x)^{\frac{p-1}{2}}(v))\end{array}$ \end{center} is a $p$-map.
\end{ex}

\end{document}